%% file: settings.tex
\documentclass[11pt, a4paper, twoside, reqno]{amsart}
\usepackage[utf8]{inputenc}
\usepackage[T1]{fontenc}
\usepackage{lmodern}
\usepackage{microtype}
\usepackage[frak=boondox]{mathalpha}
\usepackage{dsfont}
\usepackage{bbold}
\usepackage{upgreek}
\usepackage{mathrsfs}
\usepackage{euscript}
\usepackage{amssymb}
\usepackage{mathabx}
\usepackage[top=1in, bottom=1in, left=1in, right=1in]{geometry}
\usepackage{enumitem}
\usepackage{parskip}
\usepackage{float}
\usepackage{caption}
\usepackage{tabto}
\usepackage{color}
\usepackage{mdframed}
\usepackage{hyphenat}
\usepackage{comment}

\usepackage{amsmath, amssymb, amsthm, mathtools}
\usepackage{extarrows}
\usepackage{tikz-cd}
\usepackage[all,cmtip]{xy} 
\usepackage{graphicx}

\usepackage[numbers]{natbib}
\setcitestyle{open={},close={}}
\makeatletter
\renewcommand{\@biblabel}[1]{[#1]\hfill}
\makeatother

\usepackage[hidelinks]{hyperref}
\usepackage[nameinlink]{cleveref}
\newtheoremstyle{mystyle}
  {\topsep}    
  {\topsep}   
  {\itshape}   
  {}           
  {\bfseries}   
  {.}          
  {.5em}        
  {}            

\newtheoremstyle{spacedremark} 
  {\topsep}    
  {\topsep}     
  {\normalfont} 
  {}           
  {\bfseries}   
  {.}          
  {.5em}        
  {}               

\newtheorem{thm}{Theorem}[subsection]
\crefname{thm}{Theorem}{Theorems}

\newtheorem{lem}[thm]{Lemma}
\crefname{lem}{Lemma}{Lemmas}

\crefname{Q}{Question}{Questions}

\newtheorem{prop}[thm]{Proposition}
\crefname{prop}{Proposition}{Propositions}

\newtheorem{cor}[thm]{Corollary}
\crefname{cor}{Corollary}{Corollaries}

\crefname{conj}{Conjecture}{Conjectures}

\newtheorem{defn}[thm]{Definition}
\crefname{defn}{Definition}{Definitions}

\newtheorem{inner-recall-star}{Theorem}
\crefname{inner-recall-star}{Theorem}{Theorems}
\newenvironment{recall*}[1]
  {\begin{inner-recall-star}[see \Cref{#1}]}
  {\end{inner-recall-star}}

\theoremstyle{spacedremark}

\newtheorem{rem}[thm]{Remark}
\crefname{rem}{Remark}{Remarks}
\AtEndEnvironment{rem}{\hfill\qedsymbol}

\newtheorem{con}[thm]{Construction}
\crefname{con}{Construction}{Constructions}

\newtheorem{coex}[thm]{Counterexample}
\crefname{coex}{Counterexample}{Counterexamples}
\AtEndEnvironment{coex}{\hfill\qedsymbol}
\newtheorem{coex*}[thm]{Counterexamples}
\crefname{coex*}{Counterexample}{Counterexamples}

\crefname{obs}{Observation}{Observations}

\newtheorem{ex}[thm]{Example}
\crefname{ex}{Example}{Examples}

\newtheorem*{thm*}{Theorem}

\DeclareMathOperator*{\colim}{colim}

\newcommand{\dcolim}{\varinjlim}
\newcommand{\plim}{\varprojlim}
\newcommand*{\mcap}{\mathbin{\scalebox{1.1}[1.5]{\ensuremath{\cap}}}}
\newcommand*{\mcup}{\mathbin{\scalebox{1.1}[1.6]{\ensuremath{\cup}}}}
\newcommand*{\msqcup}{\mathbin{\scalebox{1}[1.6]{\ensuremath{\sqcup}}}}

\DeclareMathAlphabet{\duc}{U}{dutchcal}{m}{n}
\SetMathAlphabet{\duc}{bold}{U}{dutchcal}{b}{n}
\DeclareFontFamily{U}{BOONDOX-calo}{\skewchar\font=45 }
\DeclareFontShape{U}{BOONDOX-calo}{m}{n}{<-> s*[1.0] BOONDOX-r-calo}{}
\DeclareFontShape{U}{BOONDOX-calo}{b}{n}{<-> s*[1.0] BOONDOX-b-calo}{}
\DeclareMathAlphabet{\mal}{U}{BOONDOX-calo}{m}{n}
\SetMathAlphabet{\mal}{bold}{U}{BOONDOX-calo}{b}{n}

\newcommand{\shrp}{{\mathord{\text{\tiny $\sharp$}}}}
\newcommand{\esc}{\EuScript}

\setlist[enumerate]{leftmargin=*, nosep}
\begin{document}

\title[Schematic Functorialities of Birational Motivic Homotopy Categories]{Schematic Functorialities of \\Birational Motivic Homotopy Categories}

\author[Dipankar Maity]{Dipankar Maity}
\address{Department of Mathematical Sciences \\
Indian Institute of Science Education and Research (IISER) Mohali \\ Knowledge city, Sector 81, SAS Nagar, Manauli PO 140306 \\
India}
\email{homotopytypes@gmail.com}

\begin{abstract}
We show that pullback along morphisms that universally lift generalizations and extension along smooth morphisms make the $n$-birational motivic homotopy categories $S \mapsto \mathcal{H}^n(S)$ functorial on correspondences. For a qcqs scheme $S$ with finitely many generic points, pullback to those points identifies $\mathcal{H}^0(S)$ with the product of their birational motivic homotopy categories; in particular, $\mathcal{H}^0(V) \simeq \mathcal{H}^0(K(V))$ for every variety $V$. Over such a base $S$, a smooth morphism between smooth $S$-schemes is a $0$-birational motivic equivalence whenever its fibers over the generic points of the target are birationally contractible. We also show that a separated, quasicompact \'etale morphism to a qcqs scheme is an $n$-birational motivic equivalence over its target precisely when it is an $n$-dense open immersion, assuming the target is Noetherian when $n \ge 1$. Finally, we prove a localization sequence for equicodimensional closed immersions of Noetherian schemes in which the birational height on the closed subscheme drops by its codimension.
\end{abstract}

\maketitle

\input{body}

\phantomsection
\bibliographystyle{amsalpha}	
\renewcommand\refname{Bibliography}
\bibliography{references}
\noindent\rule{\textwidth}{0.4pt} 
\end{document}

%% file: body.tex
\tableofcontents

\section{Introduction}
\subsection{Motivation and setup}

To study the unstable counterpart of Voevodsky's slice filtration [\cite{voevodsky2002open}], Pelaez [\cite{pelaez2013birational}, \cite{pelaez2014unstable}] introduced the $n$-birational motivic homotopy categories over a field (cf.\ [\cite{choudhury2022characterisation}, \cite{cisinski2023homotopy}]), building on Levine's birational identification of the stable slice quotients [\cite{levine2008homotopy}]. In their proof of Voevodsky's slice conjectures via framed transfers, Bachmann and Elmanto [\cite{bachmann2019voevodsky}] extended this construction to general base schemes. For a qcqs scheme $S$ and $n \geq 0$, the $n$-birational motivic homotopy category $\mathcal{H}^n(S)$ is the localization of $\mathcal{H}^{\mathbb{A}^1}(S)$ at the set $B_n(S)$ of open immersions $U \hookrightarrow X$ in $Sm_S$ whose complement has codimension at least $n+1$ in $X$. Varying $n$ clearly yields a tower in $Pr^L$:
\[
\mathcal{I} \simeq \mathcal{H}^{-1}(S) \xleftarrow{L^{-1}} \mathcal{H}^0(S) \xleftarrow{L^0} \mathcal{H}^1(S) \xleftarrow{L^1} \mathcal{H}^2(S) \xleftarrow{L^2} \cdots \longleftarrow \mathcal{H}^{\mathbb{A}^1}(S).
\]

Over a perfect field, the $S^1$-stabilization of this tower recovers both Voevodsky's slice filtration and Levine's homotopy coniveau filtration [\cite{levine2008homotopy}, \cite{pelaez2013birational}, \cite{bachmann2019voevodsky}]. Over a general qcqs base scheme $S$, however, these two perspectives diverge: it is the tower of spherical nullifications of the motivic homotopy category (as studied in [\cite{asok2023p}, \S3.1]) that stabilizes to Voevodsky's slice filtration, whereas the $n$-birational localizations $\mathcal{H}^n(S)$ provide the unstable counterpart of Levine's coniveau filtration. We shall therefore call this the unstable coniveau tower (rather than `the unstable slice tower' as used by Pelaez), denoted $\mathcal{H}^-$. The aim in this paper is to study the schematic functoriality of the assignment $S \mapsto \mathcal{H}^n(S)$, and hence of $S \mapsto \mathcal{H}^-(S)$.

The transition of the functorial behavior from $\mathcal{H}^{\mathbb{A}^1}$ to $\mathcal{H}^n$ is, however, not routine. The main issue is that because generalizations need not lift under base change along an arbitrary morphism $f$ (see \S\ref{3.2}), the adjunction $f^* \dashv f_*$ does not in general descend to $\mathcal{H}^n$ at a fixed height $n$. Naturally, pulling back along suitably codimension-preserving morphisms (for example, universally generalization-lifting morphisms, see \Cref{uglt defn}) provides a way out, but other morphisms exhibit unexpected behavior. For example, pulling back the $0$-dense open $X \setminus Z \hookrightarrow X$ along a divisor $i \colon Z \hookrightarrow X$ yields the $(-1)$-dense inclusion $\emptyset \hookrightarrow Z$, so in general complement codimensions can drop under base change. 

The behavior of the exceptional pullback for closed immersions is also delicate, though for another reason. Namely, recall that the existence of $i^!: \mathcal{H}^{\mathbb{A}^1}(S)_\bullet \to \mathcal{H}^{\mathbb{A}^1}(Z)_\bullet$, as a right adjoint to $L_{mot}i_*$ for a closed immersion $i: Z \hookrightarrow S$, is governed by a universal colimit technique [\cite{hoyois2017six}, Proposition 5.1]. Unfortunately, our localizations $L^n$ are not locally cartesian [\cite{0bat}, Counterexample~3.2.15]. Hence the universal colimit technique does not apply directly; consequently, a naive transfer of the Morel-Voevodsky localization sequence to $\mathcal{H}^n$ is not immediate either. The goal of this paper is to resolve these obstructions and show that the birational motivic homotopy categories nonetheless enjoy rich and often surprising functorial properties. 

First, in a notable contrast with $\mathcal{H}^{\mathbb{A}^1}$ (which only admits fully faithful pullbacks along $\mathbb{A}^1$-projections), we show that $\mathcal{H}^n$ is invariant under birational modifications of the base: it induces an equivalence of $\infty$-categories along $n$-dense open immersions (\Cref{n dense equiv}). In fact, for any qcqs scheme $S$ with finitely many generic points, we prove that:
\[
\mathcal{H}^{0}(S) \simeq \prod_{s \in S^{(0)}} \mathcal{H}^{0}(k(s)).
\] This shows, in particular, that $\mathcal{H}^0(S)$ is an $\infty$-topos without the geometrically unibranch hypothesis of [\cite{0bat}, Theorem 3.1.13] (\Cref{bir topos on fin gen pt}). As a geometric consequence, we obtain a generic fiber contractibility criterion for birational equivalences. For instance, this shows that if the generic fibers of a smooth morphism of schemes are retract-rational, then the morphism is a birational motivic equivalence (\Cref{detect bir equiv of smooth by bir contr fibers}), a statement that goes beyond its classical geometric counterpart, namely that smooth morphisms with rational fibers are stably birational.

And second, despite the failure of local cartesianness, we show that the exceptional pullback functor $i^!$ does exist (needless to say, with a shift in birational height), provided the schemes involved are Noetherian. Rather than following the standard strategy of proving directly that $i_*$ preserves pointed birational equivalences up to the appropriate height shift (which requires local cartesianness), we establish this by exploiting the unstable motivic localization sequence.
\subsection{Key results and summary}
For a qcqs scheme $S$, we define the $n$-birational motivic homotopy category $\mathcal{H}^n(S)$ as the accessible localization $L^n:\mathcal{H}^{\mathbb{A}^1}(S)\to \mathcal{H}^n(S)$ [\cite{lurie2009higher}, Proposition 5.5.4.15] at the set of $n$-dense open immersions in $Sm_S$. The first easy thing that we observe is that on morphisms that universally lift generalizations (uglt for short), the assignment $S\mapsto \mathcal{H}^n(S)$ and $f\mapsto L^nf^*b^n$ (where $b^n: \mathcal{H}^n\subset \mathcal{P}_{nis}$) yields a presheaf $\mathcal{H}^n(-)^*:\mathrm{Sch}^{uglt,op}\to Pr^L$. On the other hand, since the localization is defined at morphisms of schemes on $Sm$ there is a co-presheaf $\mathcal{H}^n(-)_\shrp:\mathrm{Sch}^{sm}\to Pr^L$ that we get for free, namely $p\mapsto L^np_\shrp$. In \S3.1, we assemble these via smooth base change (see the last paragraph of \S\ref{notation} for the correspondence category convention we follow):

\begin{recall*}{full functoriality}\label{intro:uglt sm functoriality}
   There is a presheaf
   $$\mathcal{H}^n(-): {\mathrm{Cor}(\mathrm{Sch})_{uglt,sm}}^{op}\longrightarrow Pr^L$$
   whose restriction to vertical maps recovers $\mathcal{H}^n(-)_\shrp$, while its restriction to horizontal maps recovers $\mathcal{H}^{n}(-)^*$.
\end{recall*}

The main properties of this presheaf can be summarized as follows:

\begin{inner-recall-star}[Invariance, descent, and continuity of $\mathcal{H}^n$]\label{thm:invariance-descent-continuity}
For any qcqs scheme $S$, the pullback functoriality of the $n$-birational motivic homotopy $\infty$-category $\mathcal{H}^n(-)$ along universally generalization-lifting ($UGLT$) morphisms (\Cref{uglt defn}) satisfies the following properties:
\begin{enumerate}
    \item \emph{Birational invariance (\Cref{n dense equiv}):} For any $n$-dense open immersion $j : U \hookrightarrow S$, the pullback-pushforward adjunction
    \[
    j^{n*} : \mathcal{H}^n(S) \rightleftarrows \mathcal{H}^n(U) : j_*
    \]
    is an adjoint equivalence of $\infty$-categories, with inverse $j_{n\shrp} \simeq j_*$.
    \item \emph{Deformation invariance (\Cref{deformation invariant}):} For any nil immersion $i : S_0 \hookrightarrow S$, the pullback-pushforward adjunction $$i^{n*} : \mathcal{H}^n(S) \leftrightarrows\mathcal{H}^n(S_0):i_*$$ is an equivalence of $\infty$-categories.
    \item \emph{Nisnevich descent (\Cref{bir detect}):} The pullback presheaf $\mathcal{H}^n(-)^* : \mathrm{Sch}^{uglt,op} \to Pr^L$ satisfies Nisnevich descent.
    \item \emph{Homotopy invariance (\Cref{homotopy}):} For any Nisnevich-locally trivial $\mathbb{A}^m$-bundle (in particular, any vector or affine bundle) $\pi : E \to S$, the pullback $\pi^{n*} : \mathcal{H}^n(S) \hookrightarrow \mathcal{H}^n(E)$ is fully faithful.
    \item \emph{Flat-affine continuity (\Cref{continuity of Hn}):} If $S \simeq \plim_\alpha S_\alpha$ is a cofiltered limit of qcqs schemes along flat affine transition morphisms (assuming that either $S$ and each $S_\alpha$ are Noetherian, or that $n=0$ and the transition maps are dominant), the canonical functor
    \[
    \dcolim_\alpha \mathcal{H}^n(S_\alpha) \xrightarrow{\;\sim\;} \mathcal{H}^n(S)
    \]
    is an equivalence in $Pr^L$ (equivalently, $\mathcal{H}^n(S) \xrightarrow{\sim} \plim_\alpha \mathcal{H}^n(S_\alpha)$ in $\widehat{Cat}_\infty$).
\end{enumerate}
\end{inner-recall-star}

As a consequence of these results, we obtain in \S4.3 a nice generic decomposition at $n=0$, which in turn strengthens [\cite{0bat}, Theorem 3.1.13] by removing the geometrically unibranch hypothesis, and provides a fiberwise characterization of birational equivalences of schemes like its ordinary motivic counterpart as studied in [\cite{dubouloz2015familiesmathbba1contractibleaffinethreefolds}]:

\begin{inner-recall-star}[Generic decomposition]\label{thm:generic-decomposition}
    Let $S$ be a qcqs scheme with finitely many generic points. Then the canonical pullback functor
    \[
    \mathcal{H}^0(S)\longrightarrow \prod_{\eta\in S^{(0)}}\mathcal{H}^0(k(\eta))
    \]
    is an equivalence (\textup{\Cref{generic decomposition}}; in particular, $\mathcal{H}^0(V)\simeq\mathcal{H}^0(K(V))$ for a variety $V$). Consequently:
    \begin{enumerate}
        \item \emph{Birational topos (\Cref{bir topos on fin gen pt}):} $\mathcal{H}^0(S)$ is an $\infty$-topos.
        \item \emph{\'Etale birational equivalences (\Cref{dense open into base}):} A morphism $f:X\to Y$ in $Sm_S$ is a dense open immersion if and only if $f$ is \'etale and a $0$-birational motivic equivalence in $\mathcal{P}(Y)$.
        \item \emph{Fiberwise detection (\Cref{fiberwise criteria for smoo}):} A smooth morphism $f:X\to Y$ in $Sm_S$ is a $0$-birational motivic equivalence in $\mathcal{P}(S)$ whenever its generic fibers $X_{k(\eta)}$ ($\eta\in Y^{(0)}$) are birationally contractible.
        \item \emph{Pure transcendence invariance (\Cref{trans}):} For any generically transcendental morphism $X\to S$ of qcqs schemes with finitely many generic points, the induced pullback $\mathcal{H}^0(S)\to \mathcal{H}^{0}(X)$ is fully faithful. In particular, $\mathcal{H}^0(k)\to \mathcal{H}^0(k(t_1,\dots,t_n))$ is fully faithful for any field $k$.
    \end{enumerate}
\end{inner-recall-star}
In \S\ref{3.2} and \S4.2, we show that although non-dominant closed immersions do not preserve $n$-dense open immersions under pullback and $L^n$ is not locally cartesian, closed immersions still satisfy a complete codimension-shifted functoriality and localization formalism:

\begin{inner-recall-star}[Shifted closed functoriality and birational localization sequence]\label{thm:closed-functoriality-localization}
Let $i: Z \hookrightarrow X$ be a closed immersion of Noetherian schemes of codimension $c:=\mathrm{cod}(i)$ and supremal codimension $s=\mathrm{scod}(i)<\infty$ (\Cref{cod and scod}) with open complement $j: U \hookrightarrow X$. Then for all $n\geq 0$:
\begin{enumerate}
    \item \emph{Shifted pullback-pushforward adjunction (\Cref{closed rel d push}, \Cref{closed push full-faithful}):} $i_*$ maps $n$-birational spaces to $(n+s)$-birational spaces, inducing an adjunction with fully faithful right adjoint:
    \[
    i^{n*}:=L^{n}i^*: \mathcal{H}^{n+s}(X) \leftrightarrows \mathcal{H}^{n}(Z):i_{*}.
    \]
    \item \emph{Shifted exceptional pullback (\Cref{excp push for closed}):} On pointed motivic spaces, $i_*$ maps $n$-birational equivalences to $(n+c)$-birational equivalences, inducing an adjunction
    \[
    i_{n+c*} := L^{n+c} i_* : \mathcal{H}^{n}_\bullet(Z) \leftrightarrows \mathcal{H}^{n+c}_\bullet(X) : i^!_{n+c}
    \]
    and a fiber sequence $i_* i^!_n \esc{X} \to \esc{X} \to j_* j^{n*} \esc{X}$ in $\mathcal{H}^n_\bullet(X)$ for every $\esc{X} \in \mathcal{H}^n_\bullet(X)$.
    \item \emph{Birational localization sequence (\Cref{thm:localization-sequence}):} When $i$ is equicodimensional of codimension $c=s=d$, we obtain a fiber-cofiber sequence in $Pr^L$:
    \[
     \mathcal{H}^{n+d}(U) \xrightarrow{j_{\shrp}} \mathcal{H}^{n+d}(X) \xrightarrow{ i^{*}} \mathcal{H}^{n}(Z)
    \]
\end{enumerate}
\end{inner-recall-star}

Combining these finite-height results in \S\ref{sec:height-infty}, we show that at height $\infty$ both the codimension shift along closed immersions and the non-additivity of supremal codimension vanish over finite-dimensional (or universally catenary) Noetherian schemes, assembling the shifted pullbacks along any finite-type morphism $f:X\to Y$ into unshifted functors (with terms to be explained):
\[
\xymatrix{
\mathcal{H}^{\mathbb{A}^1}(Y)
\ar[r]^{L^\infty} \ar[d]_{L_{mot}f^*}
& \mathcal{H}^\infty(Y)
\ar[r]^{\mathrm{ev}_{n+d}} \ar[d]_{f^{\infty*}}
& \mathcal{H}^{n+d}(Y)
\ar[d]^{L^nf^*}\\
\mathcal{H}^{\mathbb{A}^1}(X)
\ar[r]_{L^\infty}
& \mathcal{H}^\infty(X)
\ar[r]_{\mathrm{ev}_n}
& \mathcal{H}^n(X)
}
\]

\begin{inner-recall-star}[Unshifted correspondence functoriality at height $\infty$, see \Cref{infty birational functoriality}]
Let $\mathrm{N}$ denote either of the full subcategories of $\mathrm{Noet}$ consisting of: (1) finite-dimensional schemes or (2) universally catenary schemes. Then the limit $\mathcal{H}^\infty(-):=\plim_n^{Pr^L}\mathcal{H}^n(-)$ extends to an unshifted correspondence presheaf
    $$\mathcal{H}^\infty(-): {\mathrm{Cor}(\mathrm{N})_{ft,sm}}^{op}\longrightarrow Pr^L$$
    whose restriction to $\mathrm{N}^{ft,op}$ is a Nisnevich sheaf lifting to $\mathrm{CAlg}(Pr^{L,\otimes})$.
\end{inner-recall-star}

\begin{rem}[On $\mathbb{P}^1$-stabilization] We work exclusively in the unstable setting (whose $S^1$-stabilization over perfect fields recovers $\mathcal{SH}^{S^1}(k)/f_{n+1}$ [\cite{csf}]). Naive $\mathbb{P}^1$-stabilization cannot be performed internally at any fixed finite height $n \geq 0$ since the $(n+1)$-fold Thom sphere $\Sigma_{\mathbb{P}^1}^{n+1} S^0 \simeq \mathrm{Th}_S(\mathbb{A}_S^{n+1}) \simeq \Sigma_{S^1}(\mathbb{A}_S^{n+1} \setminus 0_S)$ is already contractible in $\mathcal{H}_\bullet^n(S)$, forcing $\mathcal{H}_\bullet^n(S)[(\mathbb{P}_S^1)^{-1}] \simeq *$. 
\end{rem}

\subsection{Organization of the paper} With a few preliminary results on $n$-birational geometry over a base (\S\ref{2.1}) in Section 2, we set up the general structure of $\mathcal{H}^n$ and $\mathcal{H}_\bullet^n$. Section 3 is then devoted to the functorial aspect of the construction $S\mapsto \mathcal{H}^n(S)$. Finally, in Section 4 we analyze several of the properties enjoyed by the functor $\mathcal{H}^n(-)$.
\subsection{Notation and terminology}\label{notation}
Throughout this paper, all schemes are qcqs and all morphisms of schemes are separated. Smooth morphisms are assumed to be quasicompact, hence of finite presentation. By a birational morphism, we mean a span $X \hookleftarrow U \hookrightarrow Y$ of dense open immersions. A scheme $X/S$ is called rational if it is birational to $\mathbb{A}^n_S$ over $S$. We denote by $\mathrm{Sch}$ the category of qcqs schemes, by $\mathrm{Noet} \subset \mathrm{Sch}$ the full subcategory of Noetherian schemes.

We freely use the language of $\infty$-categories following [\cite{lurie2009higher}, \cite{lurie2017higher}]. Fixing Grothendieck universes as usual, we write $Cat_{\infty}$ for the $\infty$-category of large $\infty$-categories and $Pr^L \subset Cat_{\infty}$ for the subcategory of large presentable $\infty$-categories and left adjoints (with $Pr^R\simeq (Pr^L)^{op}$ for presentable $\infty$-categories and right adjoints, and $Pr^L_{\bullet}$, $Pr^R_{\bullet}$ for their pointed counterparts). We equip $Pr^L$ with the standard symmetric monoidal structure $Pr^{L,\otimes}$ [\cite{lurie2017higher}, \S4.8.1] and write $\mathrm{CAlg}(Pr^{L,\otimes})$ for the full subcategory of commutative algebra objects therein.

For an $\infty$-category $\mathcal{C}$, we denote by $\mathcal{P}(\mathcal{C})$ the $\infty$-category of presheaves of spaces on $\mathcal{C}$, and by $h\colon \mathcal{C} \to \mathcal{P}(\mathcal{C})$ the Yoneda embedding. If $\mathcal{C}$ admits finite coproducts, $\mathcal{P}_\Sigma(\mathcal{C}) \subset \mathcal{P}(\mathcal{C})$ denotes the full subcategory of presheaves that send finite coproducts to products. For a composition-closed class $M \subset Mor(\mathcal{C})$, we denote by $\mathcal{C}^M$ the wide subcategory spanned by $M$. For $S \in \mathcal{C}$, we denote by $\mathrm{M}_S \subset \mathcal{C}_{/S}$ the full subcategory spanned by $M$-morphisms over $S$, and we set $\mathrm{M}^M_S := (\mathcal{C}^M)_{/S}$.

We use the language of correspondence categories from [\cite{MR3701352}, Chapter 7, \S1], with vertical and horizontal legs swapped (so that our morphisms are read off from bottom to right). For an $\infty$-category $\mathcal{C}$ with admissible classes of morphisms $v$ (vertical) and $h$ (horizontal) stable under base change along one another, we set $\mathrm{Cor}(\mathcal{C})_{h,v}^{op} := Corr(\mathcal{C})^{isom}_{v,h}$, where a $1$-morphism from $X$ to $Y$ is a span with vertical leg $p\in v$ and horizontal leg $q\in h$, composed via pullback squares:
\[
\xymatrix{
Z \ar[r]^-{q\in h}\ar[d]_{p\in v} & Y && Z\times_Y Z' \ar[r]^-{q''}\ar[d]_{p''} & Z' \ar[r]^-{q'}\ar[d]_{p'} & W\\
X & && Z \ar[r]^-{q}\ar[d]_{p} & Y & \\
& && X & &
}
\]
Accordingly, a presheaf $\Phi:\mathrm{Cor}(\mathcal{C})_{h,v}^{op}\to Pr^L$ sends a span $X\xleftarrow{p} Z\xrightarrow{q} Y$ to $\Phi(p)\circ \Phi(q):\Phi(Y)\to \Phi(X)$ (in our setting, $p_{n\shrp}\circ q^{n*}$ for $p$ smooth and $q$ UGLT). Hence, to assemble a presheaf on the horizontal legs and a co-presheaf on the vertical legs into a presheaf on the correspondence, one must ensure that pullback squares of the above type obtain exchange isomorphisms.

For a qcqs scheme $S$, we denote by $Sm_S \subset \mathrm{Sch}_{/S}$ the essentially small full subcategory of smooth, finitely presented $S$-schemes. While one may consider the larger category $SM_S$ of all smooth $S$-schemes, restriction yields an equivalence $\mathcal{P}_{nis}(SM_S) \simeq \mathcal{P}_{nis}(Sm_S)$ [\cite{hoyois2015quadratic}, Appendix C, Proposition C.5 (3)]; since we always work Nisnevich locally, we restrict attention to $Sm_S$. Finally, we write $\mathcal{P}(S) := \mathrm{Fun}(Sm_S^{op}, Spc)$ for the $\infty$-category of $S$-spaces, and denote by $h_SX \in \mathcal{P}(S)$ the presheaf represented by $X \in Sm_S$.

\section{Recollections on \texorpdfstring{$n$}{n}-Birational Motivic homotopy theory}
\subsection{\texorpdfstring{$n$}{n}-Birational geometry}\label{2.1}
Let $X$ be a scheme. For points $x,y\in X$, we say that $y$ is a generalization of $x$, written $x\leftsquigarrow y$, if $x\neq y$ and $x\in \bar{y}$. A chain of generalizations starting at $x$ of length $n$ is a sequence $x=x_0\leftsquigarrow x_1\leftsquigarrow \cdots \leftsquigarrow x_n$. The codimension $cod_X(x)$ is the supremum of the lengths of chains of generalizations starting at $x$; equivalently, $cod_X(x)=codim(\bar{x},X)=\dim(\mathcal{O}_{X,x})$ (see [\cite[\href{https://stacks.math.columbia.edu/tag/02I3}{Definition 02I3}]{stacks-project}] and [\cite[\href{https://stacks.math.columbia.edu/tag/02IZ}{Lemma 02IZ}]{stacks-project}]). For $d\in \mathbb{N}$, we denote by $X^{(d)}$ the set of codimension $d$ points of $X$.
\begin{defn}\label[defn]{cod and scod}
For a closed immersion $i:Z\hookrightarrow X$ (or closed subset $Z\subset X$), we define its codimension $cod(i)$ and supremal codimension $scod(i)$ in $\mathbb{N}\mcup \{\infty\}$ by
$$cod_X(Z)=cod(i):=\inf_{\eta_Z\in Z^{(0)}}\{cod_X(i(\eta_Z))\}=\inf_{z\in Z}\{cod_X(i(z))\},$$
$$scod(i):=\inf\big\{d\in \mathbb{N}\mcup \{\infty\}\;\big|\; cod_X(i(z))\leq cod_Z(z)+d \text{ for all } z\in Z\big\},$$
with $\inf\emptyset = \infty$ (so $cod(\emptyset\hookrightarrow X)=\infty$ and $scod(\emptyset\hookrightarrow X)=0$).
\end{defn}
When $Z\neq \emptyset$ and $cod_Z(z)<\infty$ for all $z\in Z$ (for instance, when $Z$ is locally Noetherian), $scod(i)=\sup_{z\in Z}\big(cod_X(i(z))-cod_Z(z)\big)$. Concatenating chains of generalizations shows that $cod_X(i(z))\geq cod(i)+cod_Z(z)$ for all $z\in Z$, so when $Z\neq \emptyset$ we have
$$cod(i)\leq \sup_{\eta_Z\in Z^{(0)}}\{cod_X(i(\eta_Z))\}\leq scod(i),$$
while for composable closed immersions $W\xhookrightarrow{k} Z\xhookrightarrow{i} X$ we always have $scod(i\circ k)\leq scod(i)+scod(k)$. When $Z\neq \emptyset$ and $scod(i)<\infty$, we set $\delta(i):=scod(i)-cod(i)$, and we say that a closed immersion $i:Z\hookrightarrow X$ (allowing $Z=\emptyset$) is equicodimensional if there exists $c\in \mathbb{N}$ such that $cod_X(i(z))=cod_Z(z)+c$ for all $z\in Z$ (when $Z\neq \emptyset$, this is equivalent to $\delta(i)=0$, with $c=cod(i)=scod(i)$).

\begin{lem}\label[lem]{scod finiteness}
    Let $i:Z\hookrightarrow S$ be a closed immersion of Noetherian schemes.
    \begin{enumerate}
        \item If $S$ is catenary or finite-dimensional, then $scod(i)<\infty$.
        \item Suppose $Z\neq \emptyset$. If $S$ is catenary and locally equidimensional (e.g., $S$ catenary and irreducible, or Cohen-Macaulay \textup{[\cite[\href{https://stacks.math.columbia.edu/tag/00N9}{Lemma 00N9}]{stacks-project}]}), or if $i$ is a regular closed immersion of pure codimension, then $scod(i)=\sup_{\eta_Z\in Z^{(0)}}\{cod_S(i(\eta_Z))\}$.
        \item When $f:X\to S$ is a flat morphism from a locally Noetherian scheme $X$ and $x\in X$, then $cod_Xx=cod_Sfx+cod_{X_{fx}}x$.
        \item When $f:X\to S$ is a flat morphism from a locally Noetherian scheme $X$, the base change $i_X:Z\times_S X\hookrightarrow X$ satisfies $scod(i_X)\leq scod(i)$.
    \end{enumerate}
\end{lem}
\begin{proof}
    (1). If $\dim(S)<\infty$, then $scod(i)\leq \dim(S)<\infty$ trivially. Assume now that $S$ is catenary and $Z\neq \emptyset$ (otherwise $scod(i)=0$). Let $\{S_k\}$ be the finite set of irreducible components of $S$, and $\{W_{k,j}\}$ those of $Z\mcap S_k$. For any $z\in Z$ and any $S_k\ni i(z)$, choosing $W_{k,j}\ni i(z)$ gives $\overline{\{i(z)\}}\subset W_{k,j}\subset S_k$, so catenarity of $S$ [\cite[\href{https://stacks.math.columbia.edu/tag/02I6}{Lemma 02I6}]{stacks-project}] yields
    $$\mathrm{codim}(\overline{\{i(z)\}},S_k)=\mathrm{codim}(W_{k,j},S_k)+\mathrm{codim}(\overline{\{i(z)\}},W_{k,j})\leq \mathrm{codim}(W_{k,j},S_k)+cod_Z(z),$$
    whence $cod_S(i(z))=\max_{S_k\ni i(z)}\{\mathrm{codim}(\overline{\{i(z)\}},S_k)\}\leq cod_Z(z)+\max_{k,j}\mathrm{codim}(W_{k,j},S_k)$. By definition this yields $scod(i)\leq \max_{k,j}\mathrm{codim}(W_{k,j},S_k)<\infty$.

    (2). Under the first hypothesis, for any $z\in Z$, choose $\eta_Z\in Z^{(0)}$ with $z\in \overline{\{\eta_Z\}}$ realizing $cod_Z(z)=\mathrm{codim}(\overline{\{z\}},\overline{\{\eta_Z\}})$ and an irreducible component $S_k\ni i(\eta_Z)$. Local equidimensionality of $S$ at $i(z)$ gives $cod_S(i(z))=\mathrm{codim}(\overline{\{i(z)\}},S_k)$, so [\cite[\href{https://stacks.math.columbia.edu/tag/02I6}{Lemma 02I6}]{stacks-project}] applied to $\overline{\{i(z)\}}\subset \overline{\{i(\eta_Z)\}}\subset S_k$ yields
    $$cod_S(i(z))=\mathrm{codim}(\overline{\{i(\eta_Z)\}},S_k)+cod_Z(z)\leq cod_S(i(\eta_Z))+cod_Z(z),$$
    giving $scod(i)\leq \sup_{\eta\in Z^{(0)}}\{cod_S(i(\eta))\}$ and hence equality. For a regular closed immersion of pure codimension $c$, repeated application of [\cite[\href{https://stacks.math.columbia.edu/tag/00KW}{Lemma 00KW}]{stacks-project}] gives $cod_S(i(z))-cod_Z(z)=c$ for all $z\in Z$.

    (3). This is a plain application of [\cite[\href{https://stacks.math.columbia.edu/tag/00ON}{Lemma 00ON}]{stacks-project}] (cf.\ [\cite{bachmann2019voevodsky}, Theorem 2.1]).

    (4). Consider the Cartesian square
    \[
    \xymatrix{
    X_Z:=Z\times_S X \ar@{^(->}[r]^-{i_X}\ar[d]_{f_Z} & X \ar[d]^{f}\\
    Z \ar@{^(->}[r]_-{i} & S
    }
    \]
    For any $x\in X_Z$ with $z:=f_Z(x)$ and $s:=i(z)$, (3) gives $cod_X(i_X(x))=cod_S(s)+cod_{X_s}(i_X(x))$ and $cod_{X_Z}(x)=cod_Z(z)+cod_{(X_Z)_z}(x)$. Since $(X_Z)_z\cong X_s$, subtracting gives $cod_X(i_X(x))-cod_{X_Z}(x)=cod_S(s)-cod_Z(z)\leq scod(i)$.
\end{proof}

\begin{rem}
    On reducible non-equidimensional schemes in $\mathrm{Noet}$, $\sup_{\eta_Z\in Z^{(0)}}cod_S(i(\eta_Z))$ can be strictly smaller than $scod(i)$ and hence may fail to be subadditive: for $S=\mathbb{A}^4_k\mcup_{\{0\}}\mathbb{A}^2_k$ and $k:\{0\}\hookrightarrow Z=\mathbb{A}^2_k\xhookrightarrow{i} S$, the generic point of $Z$ has codimension $0$ in $S$, whereas $scod(i)=cod_S(0)-cod_Z(0)=4-2=2$.
\end{rem}

\begin{defn}
    An open immersion $j: U\hookrightarrow X$ of schemes is $n$-dense if $cod_X(X\setminus j(U)) > n$, i.e., every generic point of $X\setminus j(U)$ has codimension at least $n+1$ in $X$, or equivalently every point of $X$ of codimension at most $n$ lies in $j(U)$.
\end{defn}

In [\cite{bachmann2019voevodsky}, Definition 2.2], $n$-dense open immersions are called `$n$-birational morphisms'; we use `$n$-dense' to distinguish single open immersions from birational spans. Because generalizations do not lift along arbitrary morphisms, $n$-dense open immersions are not stable under general base change, but they are preserved whenever generalizations lift:

\begin{lem}\label[lem]{glt pull}
    Let $f: S\to T$ be a morphism lifting generalizations [\cite[\href{https://stacks.math.columbia.edu/tag/0063}{Definition 0063}]{stacks-project}]. If $j: U\hookrightarrow T$ is an $n$-dense open immersion, then $f^{-1}(U)\hookrightarrow S$ is also $n$-dense.
\end{lem}
\begin{proof}
    Identifying $U$ with an open subscheme of $T$, let $s\in S\setminus f^{-1}(U)$. Then $f(s)\in T\setminus U$, so $cod_T(f(s))\geq n+1$, and hence there exists a chain of generalizations $f(s)=t_0\leftsquigarrow t_1\leftsquigarrow \cdots \leftsquigarrow t_{n+1}$ in $T$ of length $n+1$. Inductively lifting generalizations along $f$ produces a chain $s=s_0\leftsquigarrow s_1\leftsquigarrow \cdots \leftsquigarrow s_{n+1}$ in $S$ (with $s_k\neq s_{k+1}$ since $t_k\neq t_{k+1}$), whence $cod_S(s)\geq n+1$. Since $s\in S\setminus f^{-1}(U)$ was arbitrary, $cod_S(S\setminus f^{-1}(U))\geq n+1$.
\end{proof}

\begin{lem}\label[lem]{dense prod}
    Let $\{X_i\to S\}_{i=1}^k$ be a finite family of morphisms universally lifting generalizations, and let $U_i\hookrightarrow X_i$ be $n_i$-dense open immersions. Then $\prod_S U_i\hookrightarrow \prod_S X_i$ is a $(\min_i n_i)$-dense open immersion (and hence lies in $B_{\min_i n_i}(S)$ whenever each $X_i\in Sm_S$).
\end{lem}
\begin{proof}
    Setting $n:=\min_i n_i$, induction and the factorization $f\times_S g=(f\times_S 1)\circ (1\times_S g)$ reduce the claim to showing that if $U\hookrightarrow X$ is $n$-dense and $Y\to S$ universally lifts generalizations, then $U\times_S Y\hookrightarrow X\times_S Y$ is $n$-dense. Since $X\times_S Y\to X$ lifts generalizations, this follows from \Cref{glt pull}.
\end{proof}

\subsection{\texorpdfstring{$n$}{n}-Birational motivic spaces}\label{2.2}
Fix a qcqs scheme $S$. Recall that a morphism $X\to Y$ in $Sm_S$ is a Nisnevich cover if it is a surjective \'etale morphism that is completely decomposed (i.e., induces an isomorphism on residue fields over every point of $Y$). Let $\mathcal{P}_{nis}(S)\subset \mathcal{P}(S)$ denote the $\infty$-topos of Nisnevich sheaves on $Sm_S$, with left-exact accessible sheafification functor $L_{nis}:\mathcal{P}(S)\to \mathcal{P}_{nis}(S)$, and let $nis(S)$ be the class of Nisnevich local equivalences (morphisms in $\mathcal{P}(S)$ inverted by $L_{nis}$). An $S$-space $\esc{X}\in \mathcal{P}(S)$ is $\mathbb{A}^1$-local if it is local with respect to the class $\mathbb{A}(S):=\{h_S(\mathbb{A}^1_X)\to h_S(X)\mid X\in Sm_S\}$ of $\mathbb{A}^1$-projections. The motivic homotopy category $\mathcal{H}^{\mathbb{A}^1}(S)\subset \mathcal{P}_{nis}(S)$ [\cite{morel19991}] is the full subcategory of $\mathbb{A}^1$-local Nisnevich sheaves, whose objects are called motivic spaces.

Following [\cite{bachmann2019voevodsky}], for each $n\in \mathbb{N}$ we let
$$B_n(S):=\{U\xhookrightarrow{j} X \in Mor(Sm_S)\mid cod_X(X\setminus j(U))\geq n+1\}$$
denote the class of $n$-dense open immersions in $Sm_S$ (identifying schemes in $Sm_S$ with their representable presheaves).

\begin{defn}
    An $S$-space $\esc{X}\in \mathcal{P}(S)$ is an $n$-birational motivic space if it is a Nisnevich sheaf and is local with respect to $\mathbb{A}(S)\mcup B_n(S)$. We denote by $\mathcal{H}^n(S)\subset \mathcal{P}(S)$ the full subcategory of $n$-birational motivic spaces.
\end{defn}

\begin{rem}
    Since every open immersion into $X\in Sm_S$ is isomorphic over $X$ to an open subscheme of $X$, the set of $n$-dense open subschemes forms a skeleton of $B_n(S)$ and generates the same localizing class; we therefore use ``$n$-dense open immersion'' and ``$n$-dense open subscheme'' interchangeably.
\end{rem}

Since $Sm_S$ is essentially small, so are $\mathbb{A}(S)$ and $B_n(S)$. By [\cite{lurie2009higher}, Proposition 5.5.4.15], all categories in the chain of full subcategories
$$\mathcal{H}^n(S)\subset \mathcal{H}^{\mathbb{A}^1}(S)\subset \mathcal{P}_{nis}(S)\subset \mathcal{P}(S)$$
are presentable $\infty$-categories, and every composite inclusion admits an accessible localization functor. We denote the localization functor for $\mathcal{H}^{\mathbb{A}^1}(S)\subset \mathcal{P}(S)$ (and $\mathcal{H}^{\mathbb{A}^1}(S)\subset \mathcal{P}_{nis}(S)$) by $L_{mot}$, and those for $\mathcal{H}^n(S)\subset \mathcal{H}^{\mathbb{A}^1}(S)$ and $\mathcal{H}^n(S)\subset \mathcal{P}(S)$ by $L^n_S$ (simplifying the notation $L^n_{bir}$ of [\cite{bachmann2019voevodsky}], which was introduced for locally Noetherian bases). Morphisms in $\mathcal{P}(S)$ inverted by $L_{mot}$ (resp.\ $L^n_S$) are called motivic equivalences (resp.\ $n$-birational motivic equivalences), and their collections are denoted by $mot(S)$ (resp.\ $bir_n(S)$).

Likewise, the full subcategories $L_{\mathbb{A}^1}\mathcal{P}(S)$ and $L^n_{dense}\mathcal{P}(S)$ of $\mathcal{P}(S)$ spanned by $\mathbb{A}(S)$-local and $B_n(S)$-local presheaves admit accessible localizations $L_{\mathbb{A}^1}$ and $L^n_{dense}$ (setting $L_{dense}:=L^0_{dense}$), whose inverted morphisms are called homotopy equivalences and $n$-dense equivalences, respectively; note that
$$\mathcal{H}^{\mathbb{A}^1}(S)=\mathcal{P}_{nis}(S)\mcap L_{\mathbb{A}^1}\mathcal{P}(S)\quad \text{and}\quad \mathcal{H}^n(S)=\mathcal{H}^{\mathbb{A}^1}(S)\mcap L^n_{dense}\mathcal{P}(S).$$
By [\cite{lurie2009higher}, Definition 5.5.4.5 and Proposition 5.5.4.15], each class of local equivalences is the strongly saturated class (closed under pushouts, two-out-of-three, and small colimits, which we call the saturation class) generated by the morphisms being localized: thus $mot(S)$ is the saturation class generated by $nis(S)$ and $\mathbb{A}(S)$ (or homotopy equivalences), while $bir_n(S)$ is generated by $mot(S)$ and $B_n(S)$ (or $n$-dense equivalences), equivalently by $nis(S)$ and $\mathbb{A}(S)\mcup B_n(S)$. We omit $S$ from the notation when the base scheme is clear from context.
\begin{rem}
When $n=0$, $\mathcal{H}^0(S)$ coincides with $\mathcal{H}^{b\mathbb{A}^1}(S)$ of [\cite{0bat}], and hence with the birational homotopy category $\mathcal{H}^b(S):=L_{dense}\mathcal{P}_\Sigma(S)$ (see Theorem 2.2.2(3) of \textit{loc.\ cit.}, recalled as \Cref{Motivic equivalences are Birational equivalences}).
\end{rem}
\begin{rem}
    By convention, we set $\mathcal{H}^{-1}(S)\simeq \{\underline{\emptyset}\to *\}=:\mathcal{I}$ (the interval category), $\mathcal{H}^{-m}(S)\simeq \{*\}$ (the terminal category) with $L^{-1}\simeq \tau_{\leq -1}^{nis}$ and $L^{-m}\simeq *$ for all $m\geq 2$.
\end{rem}
The pointed category $\mathcal{H}^n_\bullet(S):=\mathcal{H}^n(S)_{*/}\simeq \mathcal{H}^n(S)\otimes Spc_*$ (note that the terminal $S$-space $*$ is local) is also presentable [\cite{lurie2009higher}, Proposition 5.5.2.10]; the forgetful functor $\mathcal{H}^n_\bullet(S)\to \mathcal{H}^n(S)$ admits a (colimit-preserving) left adjoint $(-)_\bullet$ sending $\esc{X}\mapsto \esc{X}_+:=\esc{X}\msqcup *$ (so $h_S(X)\mapsto h_S(X)_+:=h_S(X\msqcup S)$ pointed by $S\hookrightarrow X\msqcup S$).

\begin{con}[The unstable birational tower.]\label{bst} Since $B_{n+1}(S)\subset B_n(S)$, we have natural chains of inclusions of full subcategories and saturation classes:
$$\mathcal{H}^{-m}(S)\simeq *\subset\mathcal{I}\simeq \mathcal{H}^{-1}(S)\subset \mathcal{H}^{b}(S)=\mathcal{H}^{0}(S)\subset \mathcal{H}^{1}(S)\subset \mathcal{H}^{2}(S)\subset \cdots \subset\mathcal{H}^{\mathbb{A}^1}(S),$$
$$all(S)\supset connective(S)\supset bir_0(S)\supset bir_1(S)\supset bir_2(S)\supset \cdots \supset mot(S)$$ (where $m\ge 2)$.
In particular, every $n$-birational motivic equivalence (which we also call a birational motivic equivalence of height $n$) is a $0$-birational motivic equivalence (simply called a birational motivic equivalence). Denoting the relative localizations $\mathcal{H}^{n+1}(S)\to \mathcal{H}^n(S)$ still by $L^n$, we obtain the unpointed and pointed unstable birational towers in $Pr^L$:
$$\mathcal{I}\simeq \mathcal{H}^{-1}(S)\xleftarrow{L^{-1}}\mathcal{H}^{0}(S)\xleftarrow{L^{0}} \mathcal{H}^{1}(S)\xleftarrow{L^{1}}\mathcal{H}^{2}(S)\xleftarrow{L^{2}} \cdots \longleftarrow \mathcal{H}^{\mathbb{A}^1}(S),$$
$$*\simeq \mathcal{H}_\bullet^{-1}(S)\xleftarrow{L^{-1}}\mathcal{H}_\bullet^{0}(S)\xleftarrow{L^{0}} \mathcal{H}_\bullet^{1}(S)\xleftarrow{L^{1}}\mathcal{H}_\bullet^{2}(S)\xleftarrow{L^{2}} \cdots \longleftarrow \mathcal{H}_\bullet^{\mathbb{A}^1}(S).$$
We defer the study of the height-$\infty$ limit of this tower to \S\ref{sec:height-infty}.
\end{con}

\textbf{On equivalent constructions of \texorpdfstring{$\mathcal{H}^n$}{Hn}.}
We record two alternative descriptions of $\mathcal{H}^n(S)$: first by restricting the smooth site when $S$ is separated, and second, for $n=0$, by dropping the Nisnevich and $\mathbb{A}^1$-locality conditions.

\textit{1. Changing the base site from $Sm_S$.} Let $AffSm_S\subset Sm_S$ denote the full subcategory of smooth $S$-schemes $X/S$ such that $X$ is an affine scheme, and when $S$ is semiseparated consider the intermediate full subcategories
$$AffSm_S\subset SmAff_S\subset SmQproj_S\subset Sm_S$$
spanned by smooth affine and smooth quasiprojective morphisms over $S$, respectively. All inclusions in this chain are continuous and cocontinuous for the Nisnevich topology. When $S$ is separated, every $X\in Sm_S$ is separated and hence admits a finite Zariski cover by objects of $AffSm_S$ whose finite intersections remain in $AffSm_S$; standard \v{C}ech descent therefore yields equivalences of Nisnevich $\infty$-topoi ([\cite{MR3679884}, Lemma 3.3.2]; see also [\cite[\href{https://stacks.math.columbia.edu/tag/03A0}{Lemma 03A0}]{stacks-project}]):
$$\mathcal{P}_{nis}(Sm_S)\simeq \mathcal{P}_{nis}(SmQproj_S)\simeq \mathcal{P}_{nis}(SmAff_S)\simeq \mathcal{P}_{nis}(AffSm_S).$$
Every $\mathbb{A}^1$-projection in $Sm_S$ is a Nisnevich colimit of $\mathbb{A}^1$-projections in $AffSm_S$ (see also [\cite{MR3679884}, Lemma 5.1.2]), so these equivalences restrict to $\mathcal{H}^{\mathbb{A}^1}$. To pass to $\mathcal{H}^n$, let $X\in Sm_S$ over a separated base $S$ and choose a finite Zariski cover $\{U_i\hookrightarrow X\}$ in $AffSm_S$ with $u:\bigsqcup_i U_i\to X$ and \v{C}ech nerve $c_\bullet u:\check{C}_\bullet u\to X$. For any $j:V\hookrightarrow X$ in $B_n(S)$ and $[m]\in\Delta$, form the pullback square:
\[
\xymatrix{
\check{C}_m u\times_X V \ar@{^(->}[r]^-{j_m}\ar[d] & \check{C}_m u \ar[d]^{c_m u}\\
V \ar@{^(->}[r]_-{j} & X
}
\]
Since open immersions universally lift generalizations, each $j_m$ lies in $B_n(AffSm_S)$ by \Cref{dense prod}, expressing $j\simeq \colim_{[m]\in\Delta^{op}} j_m$ in $\mathcal{P}_{nis}(Sm_S)$ as a colimit of morphisms in $B_n(AffSm_S)$. Writing
$$\mathcal{H}^n(\mathcal{C}_S):=\mathcal{P}_{nis}(\mathcal{C}_S)\mcap L_{\mathbb{A}^1}\mathcal{P}(\mathcal{C}_S)\mcap L^n_{dense}\mathcal{P}(\mathcal{C}_S)$$
for an admissible full subcategory $\mathcal{C}_S\subset Sm_S$, we obtain:

\begin{thm}\label{smqproj}
    Let $S$ be a separated scheme. Then, for every $n\geq 0$ (as well as for $\mathcal{H}^{\mathbb{A}^1}$), there are canonical equivalences of presentable $\infty$-categories:
    $$\mathcal{H}^{n}(Sm_S)\simeq \mathcal{H}^{n}(SmQproj_S)\simeq \mathcal{H}^{n}(SmAff_S)\simeq \mathcal{H}^{n}(AffSm_S).$$
\end{thm}

\textit{2. Reducing local conditions for constructing $\mathcal{H}^0$.}\label{Hb=H0} When $n=0$, several localizing conditions in the definition of $\mathcal{H}^0(S)$ become redundant (for instance, $\mathcal{H}^0(S)$ also coincides with the $0$-dense localization of the logarithmic motivic homotopy category of [\cite{binda2023logarithmic}]; see [\cite{0bat}, Corollary 2.3.6]). For the record, let us also recall another important redundancy proved in \textit{loc.\ cit.} Recall that $\mathcal{P}_\Sigma(S)\subset \mathcal{P}(S)$, the full subcategory of presheaves transforming finite disjoint unions into products, is the sifted cocompletion of $Sm_S$ [\cite{lurie2009higher}, Proposition 5.5.8.10] and coincides with the category of sheaves for the disjoint-union topology [\cite{bachmann2017norms}, Lemma 2.4]. Following [\cite{0bat}], we define the birational homotopy category by $\mathcal{H}^b(S):=L_{dense}\mathcal{P}_\Sigma(S)$ and denote by $bir_\Sigma(S)$ the associated saturation class in $\mathcal{P}(S)$.

\begin{lem}[\textup{[\cite{MR3404383}, Appendix A], [\cite{0bat}, Lemma 2.2.1]}]\label[lem]{bir local is P1 local}
    Let $S$ be a qcqs scheme. Every $0$-dense local presheaf $F:Sm_S^{op}\to \mathrm{Set}$ is $\mathbb{A}^1$-local.
\end{lem}

Combining \Cref{bir local is P1 local} with the generic splitting of Nisnevich covers over qcqs schemes yields the following comparison, established via local objects in [\cite{0bat}]:

\begin{thm}[\textup{[\cite{0bat}, Theorem 2.2.2(3)]}]\label{nis is bir}\label{Motivic equivalences are Birational equivalences}
    Let $S$ be a qcqs scheme. Then $nis(S)\mcup \mathbb{A}(S)\subset bir_\Sigma(S)$, and consequently the inclusion $\mathcal{H}^0(S)\subset \mathcal{H}^b(S)$ is an equivalence of $\infty$-categories.
\end{thm}

\begin{rem}\label[rem]{gap in CR24}
    The argument of [\cite{choudhury2022characterisation}, Lemma 3.6] (over a field $k$) claims that Nisnevich equivalences in $\mathcal{P}(Sm_k)$ already lie in the $0$-dense saturation class without passing to $\mathcal{P}_\Sigma(Sm_k)$, overlooking finite disjoint-union decompositions (for $X=A\msqcup B$, $h_S(A)\msqcup h_S(B)\to h_S(X)$ is not inverted by $L_{dense}$). A direct saturation-class proof of $nis(S)\subset bir_\Sigma(S)$ over qcqs schemes repairing this can be found in the author's note [\cite{nisautobir}].
\end{rem}

\textbf{On generating \texorpdfstring{$\mathcal{H}^n$}{Hn}.}
For a smooth scheme $X/S$, we set $h^n_S(X):=L^n_S(h_S(X))$. Except for the refinements at $n=0$, the following generation properties are standard:

\begin{prop}\label[prop]{compact generation}
    Let $S$ be a qcqs scheme.
    \begin{enumerate}
    \item $\mathcal{H}^n(S)$ is generated under sifted colimits by $\{h^n_S(X)\mid X\in Sm_S\}$. We may take $X\in AffSm_S$ when $S$ is separated, and $X$ smooth and proper when $n=0$ and $S=\mathrm{Spec}(k)$ for a field $k$ of characteristic zero.
    \item Each inclusion in $\mathcal{H}^n(S)\subset \mathcal{H}^{\mathbb{A}^1}(S)\subset \mathcal{P}_{nis}(Sm_S)\subset \mathcal{P}(Sm_S)$ is closed under filtered colimits. When $n=0$, each inclusion in $\mathcal{H}^0(S)\subset \mathcal{P}_\Sigma(Sm_S)\subset \mathcal{P}(Sm_S)$ is closed under sifted colimits; if moreover $S$ is geometrically unibranch with finitely many generic points, then $\mathcal{H}^0(S)\subset \mathcal{P}_\Sigma(Sm_S)$ and $\mathcal{H}^0(S)\subset \mathcal{P}_{nis}(Sm_S)$ are closed under all small colimits.
    \item For every $X\in Sm_S$, $h^n_S(X)$ is compact in $\mathcal{H}^n(S)$, and also projective (i.e., $\mathrm{Map}_{\mathcal{H}^0(S)}(h^0_S(X),-)$ preserves sifted colimits) when $n=0$.
    \end{enumerate}
\end{prop}
\begin{proof}
    (1). Since $\mathcal{P}_\Sigma(Sm_S)$ is generated under sifted colimits by representables and $L^n_S$ preserves colimits, $\{h^n_S(X)\mid X\in Sm_S\}$ generates $\mathcal{H}^n(S)$ under sifted colimits. The affine refinement for separated $S$ follows from \Cref{smqproj}, and the smooth proper refinement for $n=0$ in characteristic zero is [\cite{0bat}, Lemma 3.1.5].

    (2). Since $\mathcal{P}_{nis}(Sm_S)\subset \mathcal{P}(Sm_S)$ is closed under filtered colimits while $L_{\mathbb{A}^1}\mathcal{P}(S)$ and $L^n_{dense}\mathcal{P}(S)$ are closed under all small colimits in $\mathcal{P}(Sm_S)$, the first claim follows from
    $$\mathcal{H}^n(S)=\mathcal{P}_{nis}(Sm_S)\mcap L_{\mathbb{A}^1}\mathcal{P}(S)\mcap L^n_{dense}\mathcal{P}(S).$$
    When $n=0$, \Cref{Motivic equivalences are Birational equivalences} gives
    $$\mathcal{H}^0(S)=\mathcal{H}^b(S)=\mathcal{P}_\Sigma(Sm_S)\mcap L_{dense}\mathcal{P}(Sm_S),$$
    where $\mathcal{P}_\Sigma(Sm_S)\subset \mathcal{P}(Sm_S)$ is closed under sifted colimits. Under the last hypothesis on $S$, every connected object of $Sm_S$ is irreducible, so both inclusions are also closed under small coproducts (and hence all small colimits) [\cite{0bat}, Proposition 3.1.12].

    (3). Every $h_S(X)$ is compact projective in $\mathcal{P}(Sm_S)$ [\cite{lurie2009higher}, Corollary 5.5.8.17 and Definition 5.5.8.18]; since $\mathrm{Map}_{\mathcal{H}^n(S)}(h^n_S(X),-)$ restricts $\mathrm{Map}_{\mathcal{P}(Sm_S)}(h_S(X),-)$ along $\mathcal{H}^n(S)\subset \mathcal{P}(Sm_S)$, the claims follow from (2).
\end{proof}

\begin{rem}
    None of the localizations $L^n$ is locally cartesian (see [\cite{0bat}, Counterexample 3.2.15] for $n=0$), again because $n$-dense open immersions are not stable under arbitrary pullbacks.
\end{rem}

\subsection{Monoidal structures on \texorpdfstring{$\mathcal{H}^n$}{Hn} and \texorpdfstring{$\mathcal{H}^n_\bullet$}{Hn-bullet}}
Fix a qcqs scheme $S$. We lift the unpointed and pointed $n$-birational localizations over $S$ to presentably symmetric monoidal functors, deferring their schematic functorialities to \S3.

\begin{prop}\label[prop]{bir cart}
    For any $\esc{X}\in \mathcal{P}(S)$, the functor $\esc{X}\times -:\mathcal{P}(S)\to \mathcal{P}(S)$ preserves $bir_n(S)$-equivalences. Consequently, $L^n:\mathcal{P}(S)\to \mathcal{H}^n(S)$ preserves finite products, and $\esc{Y}^{\esc{X}}\in \mathcal{H}^n(S)$ whenever $\esc{Y}\in \mathcal{H}^n(S)$.
\end{prop}
\begin{proof}
    Since $\mathcal{P}(S)$ is cartesian closed and generated under colimits by $\{h_S(X)\mid X\in Sm_S\}$, it suffices to check that $h_S(X)\times -$ maps $nis(S)$, $\mathbb{A}(S)$, and $B_n(S)$ into $bir_n(S)$. Stability of $nis(S)$ and $\mathbb{A}(S)$ is immediate, and stability of $B_n(S)$ follows from \Cref{dense prod}. For any $\esc{X},\esc{Y}\in \mathcal{P}(S)$, the composite $\esc{X}\times \esc{Y}\to \esc{X}\times L^n\esc{Y}\to L^n\esc{X}\times L^n\esc{Y}$ therefore lies in $bir_n(S)$ with target in $\mathcal{H}^n(S)$, yielding $L^n(\esc{X}\times \esc{Y})\xrightarrow{\sim} L^n\esc{X}\times L^n\esc{Y}$ (while $L^n(*)\simeq *$). When $\esc{Y}\in \mathcal{H}^n(S)$, adjunction immediately implies $\esc{Y}^{\esc{X}}\in \mathcal{H}^n(S)$.
\end{proof}

\begin{thm}\label{Ln is cart}
    Let $S$ be a qcqs scheme.
    \begin{enumerate}
        \item The chain of localizations $\mathcal{P}(S)\xrightarrow{L_{mot}} \mathcal{H}^{\mathbb{A}^1}(S)\xrightarrow{L^n} \mathcal{H}^n(S)$ lifts in $\mathrm{CAlg}(Pr^{L,\otimes})$ to cartesian symmetric monoidal $\infty$-categories:
        $$\mathcal{P}(S)^\times \xrightarrow{L_{mot}^\times} \mathcal{H}^{\mathbb{A}^1}(S)^\times \xrightarrow{L^{n\times}} \mathcal{H}^n(S)^\times,$$
        whose cartesian structure on $\mathcal{H}^n(S)$ is closed via the restricted internal hom $\esc{X}\times -:\mathcal{H}^n(S)\rightleftarrows \mathcal{H}^n(S):(-)^{\esc{X}}$.
        \item Setting $(\esc{X},x)\wedge (\esc{Y},y):=\mathrm{cofib}\big((\esc{X},x)\vee (\esc{Y},y)\to (\esc{X},x)\times (\esc{Y},y)\big)$ in $\mathcal{P}_\bullet(S)$, $\wedge^{\mathbb{A}^1}:=L_{mot}\circ \wedge$, and $\wedge^n:=L^n\circ \wedge$, applying $-\otimes Spc_*$ in $\mathrm{CAlg}(Pr^{L,\otimes})$ yields a commutative diagram in $\mathrm{CAlg}(Pr^{L,\otimes})$:
        \[
        \xymatrix{
        \mathcal{P}(S)^{\times}\ar[r]^-{L_{mot}^\times} \ar[d]_{(-)_\bullet}&\mathcal{H}^{\mathbb{A}^1}(S)^\times\ar[r]^-{L^{n\times}} \ar[d]_{(-)_\bullet}&\mathcal{H}^{n}(S)^\times\ar[d]^{(-)_\bullet}\\
        \mathcal{P}(S)_\bullet^{\wedge}\ar[r]^-{L_{mot}} &\mathcal{H}_\bullet^{\mathbb{A}^1}(S)^{\wedge^{\mathbb{A}^1}}\ar[r]^-{L^n} &\mathcal{H}_\bullet^{n}(S)^{\wedge^n}
        }
        \]
        Moreover, $\wedge^n$ is closed on $\mathcal{H}^n_\bullet(S)$, with $\esc{X}\wedge^n -:\mathcal{H}^n_\bullet(S)\rightleftarrows \mathcal{H}^n_\bullet(S):(-)_\bullet^{\esc{X}}$ where $(-)_\bullet^{\esc{X}}:=\mathrm{fib}\big((-)^{\esc{X}}\to (-)\big)$ along $*\to \esc{X}$.
    \end{enumerate}
\end{thm}
\begin{proof}
    (1). By [\cite{hoyois2017six}, Proposition 3.15] and \Cref{bir cart}, both $L_{mot}$ and $L^n$ are compatible with the cartesian symmetric monoidal structures and preserve finite products, so the lift in $\mathrm{CAlg}(Pr^{L,\otimes})$ follows from [\cite{lurie2017higher}, Definition 2.2.1.6, Proposition 2.2.1.9, and Corollary 2.4.1.8], with closedness given by \Cref{bir cart}.

    (2). Tensoring (1) in $\mathrm{CAlg}(Pr^{L,\otimes})$ with the unit $Spc^\times \xrightarrow{(-)_\bullet} Spc_*^\wedge$ ([\cite{lurie2017higher}, Proposition 4.8.2.11]; see also [\cite{robalo2012noncommutative}, Corollaries 2.32, 2.34, 2.37]) produces the commutative diagram in $\mathrm{CAlg}(Pr^{L,\otimes})$, and the fiber description of the right adjoint $(-)_\bullet^{\esc{X}}$ follows by adjunction from the cofiber definition of $\wedge^n$.
\end{proof}

\section{Schematic functorialities of \texorpdfstring{$\mathcal{H}^{n}$}{Hn}}

We now analyze the functoriality of the assignment $S\mapsto \mathcal{H}^n(S)$ in $Pr^L$. Since $\mathcal{H}^{-1}(S)\simeq \mathcal{I}$ is constant, we fix an integer $n\geq 0$ throughout this section; when $n=0$, the functoriality of $\mathcal{H}^0$ obtains that of $\mathcal{H}^b$ (\S\ref{Hb=H0}).

Recall that any morphism of schemes $f:S\to T$ induces an adjunction $f^*:\mathcal{P}(T)\leftrightarrows \mathcal{P}(S):f_*$, where $f^*$ preserves colimits and sends $h_T(X)\mapsto h_S(X\times_T S)$, defining a presheaf $\mathcal{P}(-)^*:\mathrm{Sch}^{op}\to Pr^L$ (in fact, taking values in $\mathcal{LT}op_\infty$, the $\infty$-category of large $\infty$-topoi). Pullback stability of any Grothendieck topology $\sigma$ yields a localized presheaf $\mathcal{P}_\sigma(-)^*:\mathrm{Sch}^{op}\to Pr^L$, $f\mapsto L_\sigma f^*$, and in particular $\mathcal{P}_{nis}(-)^*:\mathrm{Sch}^{op}\to Pr^L$. Conversely, when $p:S\to T$ is smooth, postcomposition with $p$ on $Sm_S$ induces a colimit-preserving left adjoint $p_\shrp:\mathcal{P}(S)\to \mathcal{P}(T)$ to $p^*$, characterized by $h_S(X)\mapsto h_T(X)$; since $p_\shrp$ preserves Nisnevich covers, we obtain co-presheaves $\mathcal{P}(-)_\shrp,\mathcal{P}_{nis}(-)_\shrp:\mathrm{Sch}^{sm}\to Pr^L$ sending $p\mapsto p_\shrp$ and $p\mapsto L_{nis}p_\shrp$, respectively.

By smooth base change ([\cite{khan2016motivic}, \S6.1]), these assemble on the $(2,1)$-category of correspondences $\mathrm{Cor}(\mathrm{Sch})_{all,sm}$ [\cite{MR3701352}, Chapter 7, Theorem 3.2.2], and because $\mathbb{A}^1$-projections are stable under both $f^*$ (for arbitrary $f$) and $p_\shrp$ (for smooth $p$), they localize further to $\mathcal{H}^{\mathbb{A}^1}$:

\begin{prop}\label[prop]{motivic corr presheaf}
    The assignments $S\mapsto \mathcal{P}_{nis}(S)$ and $S\mapsto \mathcal{H}^{\mathbb{A}^1}(S)$ extend to presheaves
    $$\mathcal{P}_{nis}(-),\; \mathcal{H}^{\mathbb{A}^1}(-):\mathrm{Cor}(\mathrm{Sch})_{all,sm}^{op}\longrightarrow Pr^L$$
    whose restrictions to horizontal morphisms recover $\mathcal{P}_{nis}(-)^*$ ($f\mapsto L_{nis}f^*$) and $\mathcal{H}^{\mathbb{A}^1}(-)^*$ ($f\mapsto L_{mot}f^*$), and whose restrictions to vertical morphisms recover $\mathcal{P}_{nis}(-)_\shrp$ ($p\mapsto L_{nis}p_\shrp$) and $\mathcal{H}^{\mathbb{A}^1}(-)_\shrp$ ($p\mapsto L_{mot}p_\shrp$).
\end{prop}

\subsection{Pullback along morphisms lifting generalizations}

Unlike $\mathbb{A}^1$-projections, $n$-dense open immersions are not stable under arbitrary pullbacks, as density and complement codimension may fail to be preserved. Stability holds along the following class of morphisms:

\begin{defn}\label[defn]{uglt defn}
A morphism of schemes is universally generalization lifting (UGLT) if every base change of it lifts generalizations.
\end{defn}
\begin{ex}
Examples include universally open morphisms [\cite[\href{https://stacks.math.columbia.edu/tag/0407}{Lemma 0407}]{stacks-project}] (for example, flat morphisms [\cite[\href{https://stacks.math.columbia.edu/tag/03HV}{Lemma 03HV}]{stacks-project}] and universal homeomorphisms). Conversely, finitely presented uglt morphisms are necessarily universally open [\cite[\href{https://stacks.math.columbia.edu/tag/01U1}{Lemma 01U1}]{stacks-project}]. Clearly, UGLT morphisms are closed under composition.
\end{ex}
Following \S\ref{notation}, $\mathrm{Sch}^{uglt}\subset \mathrm{Sch}$ denotes the wide subcategory of UGLT morphisms; for $S\in \mathrm{Sch}$, $\mathrm{Uglt}_S\subset \mathrm{Sch}_{/S}$ is the full subcategory of UGLT $S$-schemes, and $\mathrm{Uglt}_S^{uglt}:=(\mathrm{Sch}^{uglt})_{/S}$. Because $\mathrm{Sch}^{uglt}$ lacks a terminal object, it cannot be recovered as a slice category $\mathrm{Uglt}_S$ for some $S$, so we formulate absolute statements on $\mathrm{Sch}^{uglt}$ and relative module statements on $\mathrm{Uglt}_S$.

\begin{lem}\label[lem]{ugltpull}
    Let $f: S\to T$ be a UGLT morphism of schemes. Then pullback along $f$ maps $B_n(T)$ into $B_n(S)$.
\end{lem}
\begin{proof}
    For any $j:U\hookrightarrow X$ in $B_n(T)$, the base change $X\times_T S\to X$ lifts generalizations, so $U\times_T S\hookrightarrow X\times_T S$ lies in $B_n(S)$ by \Cref{glt pull}.
\end{proof}

\begin{prop}\label[prop]{generic push}
    Let $f: S\to T$ be a UGLT morphism of schemes. Then $f^*:\mathcal{P}(T)\to \mathcal{P}(S)$ preserves $n$-birational motivic equivalences, and dually $f_*:\mathcal{P}(S)\to \mathcal{P}(T)$ preserves $n$-birational motivic spaces, inducing an adjunction
    $$L^nf^{*}:\mathcal{H}^{n}(T)\leftrightarrows \mathcal{H}^{n}(S): f_*.$$
    Moreover, $f_*:\mathcal{H}^n(S)\to \mathcal{H}^n(T)$ preserves filtered colimits (and sifted colimits when $n=0$).
\end{prop}
\begin{proof}
    Since $f^*$ preserves colimits, $nis$, and $\mathbb{A}^1$-projections, and maps $B_n(T)$ into $B_n(S)$ by \Cref{ugltpull}, it maps $bir_n(T)$ into $bir_n(S)$. By adjunction, $f_*$ preserves $\mathcal{H}^n$, giving $L^nf^*\dashv f_*$. For $X\in Sm_T$, $f^*(h_T(X)\to h^n_T(X))\in bir_n(S)$ yields $L^nf^*h^n_T(X)\simeq h^n_S(X\times_T S)$; since $\{h^n_T(X)\mid X\in Sm_T\}$ generates $\mathcal{H}^n(T)$ under colimits and consists of compact (and, when $n=0$, projective) objects (\Cref{compact generation}), $f_*$ preserves filtered (resp.\ sifted) colimits by adjunction.
\end{proof}

\begin{cor}\label[cor]{flat pull push}
    If $f:S\to T$ is flat, then $f^*$ preserves $n$-birational motivic equivalences, inducing $L^nf^{*}:\mathcal{H}^{n}(T)\leftrightarrows \mathcal{H}^{n}(S): f_*$.
\end{cor}
Because the forgetful functor $Ring_R\to Mod_R$ creates filtered colimits and flat $R$-modules are closed under filtered colimits, $\mathrm{Flat}_S\subset \mathrm{Sch}_S$ is closed under cofiltered limits with affine transition maps. In particular, for every $s\in S$, factoring through an affine open neighborhood of $s$, the morphisms $\mathrm{Spec}(\mathcal{O}_{S,s}^{sh})\to \mathrm{Spec}(\mathcal{O}_{S,s}^h)\to \mathrm{Spec}(\mathcal{O}_{S,s})\to S$ are flat. So, as a special case of \Cref{flat pull push} we have:
\begin{cor}\label[cor]{stalks of n bir equiv}
    If $\esc{X}\to \esc{Y}$ in $\mathcal{P}(S)$ is an $n$-birational motivic equivalence, then for every $s\in S$ its base change to $\mathrm{Spec}(\mathcal{O}_{S,s})$ (resp.\ $\mathrm{Spec}(\mathcal{O}_{S,s}^h)$ or $\mathrm{Spec}(\mathcal{O}_{S,s}^{sh})$) is an $n$-birational motivic equivalence.
\end{cor}

\begin{prop}\label[prop]{smooth extension pull}
    Let $p:R\to S$ be a smooth morphism of schemes. Then $p_\shrp:\mathcal{P}(R)\to \mathcal{P}(S)$ (and $L_{mot}p_\shrp:\mathcal{H}^{\mathbb{A}^1}(R)\to \mathcal{H}^{\mathbb{A}^1}(S)$) preserves $n$-birational motivic equivalences, smooth pullback commutes with localization via $L^np^*\simeq p^*L^n$, and we have adjunctions
    $$L^np_{\shrp} :\mathcal{H}^{n}(R)\leftrightarrows \mathcal{H}^{n}(S): p^*, \qquad p^*:\mathcal{H}^n(S)\leftrightarrows \mathcal{H}^n(R):p_*.$$
\end{prop}
\begin{proof}
    Since $p_\shrp$ preserves colimits, Nisnevich covers, $\mathbb{A}^1$-projections, and $n$-dense open immersions ($p_\shrp(B_n(R))\subset B_n(S)$), it maps $bir_n(R)$ into $bir_n(S)$. By adjunction, $p^*$ preserves $n$-birational motivic spaces, giving the first adjunction. Since $p$ is smooth and hence UGLT, $p^*$ also preserves $bir_n$-equivalences (\Cref{generic push}); for any $\esc{X}\in \mathcal{P}(S)$, the morphism $p^*\esc{X}\to p^*L^n\esc{X}$ therefore lies in $bir_n(R)$ with target in $\mathcal{H}^n(R)$, yielding $L^np^*\esc{X}\simeq p^*L^n\esc{X}$. Restricting \Cref{generic push} then gives the second adjunction.
\end{proof}

\begin{rem}
    When $n=0$, \Cref{bir local is P1 local} and \Cref{nis is bir} simplify the saturation arguments above: once $0$-dense open immersions are preserved, it suffices to check preservation of finite-coproduct ($\bigsqcup$-)equivalences rather than $\mathbb{A}^1$-projections and Nisnevich covers.
\end{rem}
Let us uniformize our notations for various base change functors introduced so far:
\begin{defn}
    For any functor $F:\mathcal{P}(T)\to \mathcal{P}(S)$, define its $n$-birational motivic derivation by $F(n):=L^n_S F b^n_T:\mathcal{H}^n(T)\to \mathcal{H}^n(S)$, where $b^n_T:\mathcal{H}^n(T)\hookrightarrow \mathcal{P}(T)$ is the inclusion. For a morphism $f:S\to T$, we set
    $$f^{n*}:=f^*(n)=L^n_S f^* b^n_T:\mathcal{H}^n(T)\to \mathcal{H}^n(S), \qquad f_{n*}:=f_*(n)=L^n_T f_* b^n_S:\mathcal{H}^n(S)\to \mathcal{H}^n(T),$$
    and, when $f$ is smooth, $f_{n\shrp}:=f_\shrp(n)=L^n_T f_\shrp b^n_S:\mathcal{H}^n(S)\to \mathcal{H}^n(T)$.
\end{defn}

\begin{rem}\label[rem]{pull adjunction}
    When $f:S\to T$ is UGLT and $p:R\to S$ is smooth, \Cref{generic push} and \Cref{smooth extension pull} yield $f_{n*}\simeq {f_*}_{|\mathcal{H}^n(S)}$, $p^{n*}\simeq {p^*}_{|\mathcal{H}^n(S)}$, and $p_{n\shrp}\dashv p^{n*}\dashv p_{n*}$. Furthermore, there are commutative squares
    \[
    \xymatrix{
    \mathcal{P}(T) \ar[r]^-{f^*}\ar[d]_{L^n_T} & \mathcal{P}(S) \ar[d]^{L^n_S} && \mathcal{P}(R) \ar[r]^-{p_\shrp}\ar[d]_{L^n_R} & \mathcal{P}(S) \ar[d]^{L^n_S} && \mathcal{P}(S) \ar[r]^-{p^*}\ar[d]_{L^n_S} & \mathcal{P}(R) \ar[d]^{L^n_R}\\
    \mathcal{H}^n(T) \ar[r]_-{f^{n*}} & \mathcal{H}^n(S) && \mathcal{H}^n(R) \ar[r]_-{p_{n\shrp}} & \mathcal{H}^n(S) && \mathcal{H}^n(S) \ar[r]_-{p^{n*}} & \mathcal{H}^n(R)
    }
    \] In particular, the colimit-preserving functors $f^{n*}$ and $p_{n\shrp}$ are uniquely characterized on representables by $f^{n*}h^n_T(X)\simeq h^n_S(X\times_T S)$ ($X\in Sm_T$) and $p_{n\shrp}h^n_R(W)\simeq h^n_S(W)$ ($W\in Sm_R$).
\end{rem}

\begin{thm}[Smooth base change]\label{smooth base change}
    For any pullback square of schemes on the left with $p,p'$ smooth and $q,q'$ UGLT, the induced categorical squares on the right commute up to canonical equivalence, i.e., the exchange transformations
    \begin{flalign}
        p'_{n\shrp} q'^{n*}\to q^{n*}p_{n\shrp} &\quad\text{from }\mathcal{H}^n(Y)\text{ to }\mathcal{H}^n(U),\label{sbc1}\\
        p^{n*}q_{n*}\to q'_{n*}p'^{n*} &\quad\text{from }\mathcal{H}^n(U)\text{ to }\mathcal{H}^n(Y)
    \end{flalign}
    are equivalences:
    \[
    \xymatrix{
    V\ar[r]^-{q'}\ar[d]_{p'} & Y\ar[d]^{p} && \mathcal{H}^n(Y) \ar[r]^-{q'^{n*}}\ar[d]_{p_{n\shrp}} & \mathcal{H}^n(V) \ar[d]^{p'_{n\shrp}} && \mathcal{H}^n(U) \ar[r]^-{q_{n*}}\ar[d]_{p'^{n*}} & \mathcal{H}^n(X) \ar[d]^{p^{n*}}\\
    U \ar[r]_-{q} & X && \mathcal{H}^n(X) \ar[r]_-{q^{n*}} & \mathcal{H}^n(U) && \mathcal{H}^n(V) \ar[r]_-{q'_{n*}} & \mathcal{H}^n(Y)
    }
    \]
\end{thm}
\begin{proof}
    The second map is the adjoint mate of \eqref{sbc1}, so it suffices to treat \eqref{sbc1}. Since $p'_\shrp$ and $q^*$ preserve $bir_n$-equivalences (\Cref{smooth extension pull} and \Cref{generic push}), we have $p'_{n\shrp}q'^{n*}=L^np'_\shrp L^nq'^* b^n_Y\simeq L^n(p'_\shrp q'^* b^n_Y)$ and $q^{n*}p_{n\shrp}=L^nq^*L^np_\shrp b^n_Y\simeq L^n(q^*p_\shrp b^n_Y)$. Thus \eqref{sbc1} is $L^n$ applied to the presheaf exchange transformation $p'_\shrp q'^*\to q^*p_\shrp:\mathcal{P}(Y)\to \mathcal{P}(U)$, which is an equivalence because both sides preserve colimits and send $h_Y(W)$ ($W\in Sm_Y$) to $h_U(V\times_Y W)\cong h_U(U\times_X W)$.
\end{proof}

\begin{cor}\label[cor]{fully faithful j_*}
    For an open immersion $j:U\hookrightarrow X$, both $j_{n\shrp}:\mathcal{H}^n(U)\to \mathcal{H}^n(X)$ and $j_{n*}\simeq {j_*}_{|\mathcal{H}^n(U)}:\mathcal{H}^n(U)\to \mathcal{H}^n(X)$ are fully faithful.
\end{cor}
\begin{proof}
    Since $j$ is a monomorphism, applying \Cref{smooth base change} \eqref{sbc1} to the pullback square of $j$ along itself shows that the unit $id\to j^{n*}j_{n\shrp}$ is an equivalence, so $j_{n\shrp}$ is fully faithful; since $j_{n\shrp}\dashv j^{n*}\dashv j_{n*}$, the right adjoint $j_{n*}$ is also fully faithful. Indeed, composition yields the adjunction $j^{n*}j_{n\shrp}\dashv j^{n*}j_{n*}$ where the first composite is already an equivalence.
\end{proof}

Before we state the next result, let us mention a subtle aspect of $\mathbb{A}^1$-localization: the presheaf $\mathcal{H}^{\mathbb{A}^1}(-)^*$ of $\mathbb{A}^1$-invariant objects is not $\mathbb{A}^1$-invariant itself, i.e., the pullback functor $\mathcal{H}^{\mathbb{A}^1}(S)\to \mathcal{H}^{\mathbb{A}^1}(\mathbb{A}^1_S)$ is not an equivalence. Indeed, for $p:\mathbb{A}^1_k\to \mathrm{Spec}(k)$, pulling back the nonempty $\mathbb{A}^1$-rigid sheaf $h_{\mathbb{A}^1_k}(\mathbb{G}_{m,k})$ (given by the canonical identification $\mathbb{G}_{m,k}=\mathbb{A}^1_{k}\setminus 0$) along $0:\mathrm{Spec}(k)\to \mathbb{A}^1_k$ yields $h_k(\emptyset)$, so $p^*:\mathcal{H}^{\mathbb{A}^1}(k)\to \mathcal{H}^{\mathbb{A}^1}(\mathbb{A}^1_k)$ cannot be an equivalence. What works is the following:

\begin{rem}\label[rem]{motivic category is a1 faithful}
    Nevertheless, $p^*:\mathcal{H}^{\mathbb{A}^1}(S)\to \mathcal{H}^{\mathbb{A}^1}(\mathbb{A}^1_S)$ is fully faithful for any scheme $S$: the counit $L_{mot}p_\shrp p^*\to id$ preserves colimits and sends each generator $h_S(X)$ ($X\in Sm_S$) to the $\mathbb{A}^1$-projection $L_{mot}(h_S(\mathbb{A}^1_X)\to h_S(X))$, which is an equivalence.
\end{rem}

By contrast, $\mathcal{H}^i(-)$ is genuinely invariant under $n$-dense open immersions on the base for all $i\leq n$:

\begin{thm}\label{n dense equiv}
    For an $n$-dense open immersion $j:V\hookrightarrow S$ and every $i\leq n$, the functor $j_{i\shrp}:\mathcal{H}^i(V)\to \mathcal{H}^i(S)$ is an equivalence, and hence so are $j^{i*}\simeq {j^*}_{|\mathcal{H}^i(S)}$ and $j_{i*}\simeq {j_*}_{|\mathcal{H}^i(V)}$.
\end{thm}
\begin{proof}
    By \Cref{fully faithful j_*}, it suffices to show that the counit $\varepsilon:j_{i\shrp}j^{i*}\to id_{\mathcal{H}^i(S)}$ is an equivalence. Since both functors preserve colimits, we may test this on generators $h^i_S(W)$ for smooth $p:W\to S$. Consider the pullback square:
    \[
    \xymatrix{
    W_V \ar@{^(->}[r]^-{j'}\ar[d]_{p_V} & W \ar[d]^{p}\\
    V \ar@{^(->}[r]_-{j} & S
    }
    \]
    By \Cref{pull adjunction}, $\varepsilon_{h^i_S(W)}$ identifies with $L^i_S(j':W_V\hookrightarrow W)$; since $p$ is smooth (hence UGLT) and $j$ is $i$-dense, $j'\in B_i(S)$ by \Cref{ugltpull}, so $L^i(j')$ is an equivalence.
\end{proof}

\begin{cor}[Dense locality of $\mathcal{H}^0$]\label[cor]{dense equiv}
    For a dense open immersion $j:V\hookrightarrow S$, the functor $j_{0\shrp}=L^0j_\shrp:\mathcal{H}^0(V)\to \mathcal{H}^0(S)$ is an equivalence, and hence so are $j_{0*}\simeq {j_*}_{|\mathcal{H}^0(V)}$ and $j^{0*}\simeq {j^*}_{|\mathcal{H}^0(S)}$ (and likewise for $\mathcal{H}^b=L_{dense}\mathcal{P}_\Sigma$ by \Cref{Motivic equivalences are Birational equivalences}).
\end{cor}

\begin{cor}\label[cor]{bir equiv on a dense open}
    A morphism $f:\esc{X}\to \esc{Y}$ in $\mathcal{P}(S)$ is an $n$-birational motivic equivalence if and only if there exists an $n$-dense open subscheme $V\subset S$ such that $f_V\in bir_n(V)$.
\end{cor}
\begin{proof}
    For an $n$-dense open immersion $j: V \hookrightarrow S$, $L^n_S f$ corresponds under the equivalence $j^{n*}:\mathcal{H}^n(S)\xrightarrow{\sim}\mathcal{H}^n(V)$ of \Cref{n dense equiv} to $j^{n*}L^n_S f\simeq L^n_V(f_V)$.
\end{proof}

\begin{cor}[Clopen decomposition for $\mathcal{H}^0$]\label[cor]{clopen decomp}
    Let $S$ be a qcqs scheme and $j: U \hookrightarrow S$ a quasicompact open immersion whose open exterior $k: V := S \setminus \overline{U} \hookrightarrow S$ is also quasicompact (which holds automatically if $S$ is Noetherian). Then $U \msqcup V \hookrightarrow S$ is $0$-dense, inducing a product decomposition
    $$(j^{0*}, k^{0*}): \mathcal{H}^0(S) \xrightarrow{\;\simeq\;} \mathcal{H}^0(U) \times \mathcal{H}^0(S \setminus \overline{U}),$$
    and in the pointed category $\mathcal{H}^0_\bullet(S)$ a canonical equivalence $j_{0\shrp} \simeq j_{0*}: \mathcal{H}^0_\bullet(U) \hookrightarrow \mathcal{H}^0_\bullet(S)$.
\end{cor}
\begin{proof}
    Since $\overline{U\msqcup V}=\overline{U}\mcup \overline{S\setminus \overline{U}}=S$, the open immersion $U\msqcup V\hookrightarrow S$ is $0$-dense. Combining \Cref{dense equiv} with the site decomposition $\mathcal{H}^0(U\msqcup V)\simeq \mathcal{H}^0(U)\times \mathcal{H}^0(V)$ (induced by $Sm_{U\msqcup V}\simeq Sm_U\times Sm_V$; see also \Cref{zar add} below) yields the first equivalence, and passing to pointed objects identifies both the left and right adjoints to $(j^{0*})_\bullet$ with $\esc{X}\mapsto (\esc{X},*)$.
\end{proof}

\textbf{Presheaf description of $\mathcal{H}^{n}(-)$.}
By \Cref{generic push} and \Cref{smooth extension pull}, the right adjoints $f_{n*}\simeq {f_*}_{|\mathcal{H}^n}$ on $\mathrm{Sch}^{uglt}$ and $p^{n*}\simeq {p^*}_{|\mathcal{H}^n}$ on $\mathrm{Sch}^{sm,op}$ restrict, respectively, $\mathcal{P}_{nis}(-)_*:\mathrm{Sch}\to Pr^R$ and $\mathcal{P}_{nis}(-)^*:\mathrm{Sch}^{sm,op}\to Pr^R$ to the level of $\mathcal{H}^n$. Dually, for composable morphisms $R\xrightarrow{f} S\xrightarrow{p} T$ with $f$ UGLT (resp.\ $p$ smooth), \Cref{generic push} (resp.\ \Cref{smooth extension pull}) implies that $f^*p^*\to f^*L^np^*$ (resp.\ $p_{\shrp}f_\shrp \to p_{\shrp}L^nf_\shrp$) (objectwise) lies in $bir_n(R)$ (resp.\ $bir_n(T)$), yielding $f^{n*}\circ p^{n*}\simeq (p\circ f)^{n*}$ (resp.\ $p_{n\shrp}f_{n\shrp}\simeq (p\circ f)_{n\shrp}$); passing to left adjoints via $(Pr^R)^{op}\simeq Pr^L$ produces the following functors:

\begin{defn}
    Let $\mathcal{H}^n(-)^*:\mathrm{Sch}^{uglt,op}\longrightarrow Pr^L$ denote the presheaf sending $S\mapsto \mathcal{H}^n(S)$ and $f\mapsto f^{n*}$, and let $\mathcal{H}^n(-)_\shrp:\mathrm{Sch}^{sm}\longrightarrow Pr^L$ denote the co-presheaf sending $S\mapsto \mathcal{H}^n(S)$ and $p\mapsto p_{n\shrp}$.
\end{defn}

\begin{lem}\label[lem]{pull back is cart}
    The presheaf $\mathcal{H}^{n}(-)^*: \mathrm{Sch}^{uglt, op}\longrightarrow Pr^L$ lifts to a cartesian symmetric monoidal presheaf $\mathcal{H}^{n}(-)^{\times,*}: \mathrm{Sch}^{uglt,op}\longrightarrow \mathrm{CAlg}(Pr^{L,\otimes})$.
\end{lem}
\begin{proof}
    For any UGLT morphism $f:S\to T$, $f^{n*}=L^n_S f^* b^n_T$ preserves finite products because $b^n_T$ is a right adjoint, $f^*$ is left exact, and $L^n_S$ preserves finite products (\Cref{Ln is cart}); the coherent lift to $\mathrm{CAlg}(Pr^{L,\otimes})$ follows from [\cite{lurie2017higher}, Corollary 2.4.1.8].
\end{proof}

\begin{cor}\label[cor]{cart pull functoriality}
    For any scheme $T$, $\mathcal{H}^n(-)^{\times,*}$ induces an $\mathcal{H}^n(T)^\times$-module presheaf
    $$\mathcal{H}^{n}(-)^{\times,*}:\mathrm{Uglt}_T^{uglt,op}\longrightarrow \mathrm{Mod}_{\mathcal{H}^{n}(T)^{\times}}(Pr^{L,\otimes}).$$
\end{cor}

\begin{rem}
    Unlike $\mathcal{H}^n(-)^*$, the co-presheaf $\mathcal{H}^n(-)_\shrp$ is generally only oplax symmetric monoidal [\cite{lurie2017higher}, Corollary 7.3.2.7], since $p_{n\shrp}$ need not preserve finite products. On open immersions it preserves nonempty finite products:
\end{rem}

\begin{prop}\label[prop]{j_sharp is cart}
    For an open immersion $j: U\hookrightarrow X$, the functor $j_{n\shrp}:\mathcal{H}^n(U)\hookrightarrow \mathcal{H}^n(X)$ preserves nonempty finite products and lifts to a non-unital symmetric monoidal embedding $j_{n\shrp}:\mathcal{H}^n(U)^\times \hookrightarrow \mathcal{H}^n(X)^\times$ (which is unital when $j$ is $n$-dense).
\end{prop}
\begin{proof}
    As in [\cite{khan2016motivic}, Proposition 6.1.4], since $L^n$ preserves finite products (\Cref{Ln is cart}), it suffices to check that $j_\shrp:\mathcal{P}(U)\to \mathcal{P}(X)$ preserves binary products (cf.\ [\cite{lurie2017higher}, Corollary 2.4.1.8 and \S5.4.4]). Since $j_\shrp$ and $-\times -$ preserve colimits in each variable, we may test this on representables $(h_U(W),h_U(W'))$ for $W,W'\in Sm_U$, where $W\times_U W'\cong W\times_X W'$ because $j$ is a monomorphism. When $j$ is $n$-dense, $j_{n\shrp}(*_U)=h^n_X(U)\simeq *_X$ as well.
\end{proof}
\begin{rem}
   We shall in fact show that $j_{n\shrp}$ is unital if and only if $j$ is $n$-dense, subject to the condition that either $n=0$ or $X$ is Noetherian (\Cref{n dense open into base}).
\end{rem}
\begin{prop}[Smooth projection formula]\label[prop]{sm proj formula}
    Let $p: Y\to X$ be a smooth morphism, $\esc{X},\esc{Z}\in \mathcal{H}^n(X)$, and $\esc{Y}\in \mathcal{H}^n(Y)$. Then the canonical morphisms
    $$p_{n\shrp}(p^{n*}\esc{X}\times\esc{Y})\longrightarrow \esc{X}\times p_{n\shrp}\esc{Y}, \quad p^{n*}(\esc{Z}^{\esc{X}}) \longrightarrow (p^{n*}\esc{Z})^{p^{n*}\esc{X}} $$
    are equivalences in $\mathcal{H}^n(X)$ and $\mathcal{H}^n(Y)$ respectively.
\end{prop}
\begin{proof}
    The two maps are adjoint mates, so it suffices to check the first. By \Cref{Ln is cart} and \Cref{pull adjunction}, this reduces to the presheaf projection formula $p_\shrp(p^*\esc{X}\times \esc{Y})\xrightarrow{\sim}\esc{X}\times p_\shrp\esc{Y}$, which holds because both sides preserve colimits in each variable and agree on representables $(h_X(W),h_Y(W'))$ via $(W\times_X Y)\times_Y W'\cong W\times_X W'$.
\end{proof}

\begin{cor}\label[cor]{cart extension functoriality}
    For any scheme $X$, the assignment $p\mapsto p_{n\shrp}$ lifts $\mathcal{H}^n(-)_\shrp$ to an $\mathcal{H}^n(X)^\times$-module co-presheaf
    $$\mathcal{H}^{n}(-)_\shrp: \mathrm{Uglt}_X^{sm}\longrightarrow \mathrm{Mod}_{\mathcal{H}^n(X)^\times}(Pr^{L,\otimes}).$$
\end{cor}
\begin{proof}
    For any smooth $X$-morphism $r:Y\to Z$ between schemes $p:Y\to X$ and $q:Z\to X$ in $\mathrm{Uglt}_X$, $p^{n*}\simeq r^{n*}q^{n*}$ (since $r^{*}$ preserves $n$-birational motivic equivalences). \Cref{sm proj formula} then shows that the oplax $\mathcal{H}^n(X)^\times$-linear structure map $r_{n\shrp}(p^{n*}\esc{X}\times \esc{Y})\simeq r_{n\shrp}(r^{n*}q^{n*}\esc{X}\times \esc{Y})\to q^{n*}\esc{X}\times r_{n\shrp}\esc{Y}$ induced by adjunction is an equivalence.
\end{proof}
\begin{rem}
One can similarly write down a pointed version of the above result; see [\cite{khan2016motivic}, \S6.1] for the relevant motivic counterpart (written for derived schemes though), details omitted.
\end{rem}
Let $\mathrm{Cor}(\mathrm{Sch})_{uglt,sm}$ denote the $(2,1)$-category of spans on $\mathrm{Sch}$ with smooth vertical morphisms and UGLT horizontal morphisms (note that birational morphisms $X\hookleftarrow U\hookrightarrow Y$ are actual morphisms in $\mathrm{Cor}(\mathrm{Sch})_{uglt,sm}$).

\begin{thm}\label{full functoriality}
    There is a presheaf
    $$\mathcal{H}^n(-): {\mathrm{Cor}(\mathrm{Sch})_{uglt,sm}}^{op}\longrightarrow Pr^L$$
    whose restriction to vertical morphisms recovers $\mathcal{H}^n(-)_\shrp:\mathrm{Sch}^{sm}\to Pr^L$ and whose restriction to horizontal morphisms recovers $\mathcal{H}^n(-)^*:\mathrm{Sch}^{uglt,op}\to Pr^L$. Moreover, for any qcqs scheme $S$, this lifts on $\mathrm{Uglt}_S$ to an $\mathcal{H}^n(S)^\times$-module presheaf
    $$\mathcal{H}^n(-)_{\mathcal{H}^n(S)/}:\mathrm{Cor}(\mathrm{Uglt}_S)_{uglt,sm}^{op}\longrightarrow \mathrm{Mod}_{\mathcal{H}^n(S)^\times}(Pr^{L,\otimes})$$
    recovering \Cref{cart extension functoriality} on vertical morphisms and \Cref{cart pull functoriality} on horizontal morphisms.
\end{thm}
\begin{proof}
    The first claim follows from \Cref{smooth base change} and [\cite{MR3701352}, Chapter 7, Theorem 3.2.2]. Over $S$, \Cref{cart pull functoriality} and \Cref{cart extension functoriality} lift the horizontal and vertical functors to $\mathrm{Mod}_{\mathcal{H}^n(S)^\times}(Pr^{L,\otimes})$, where the exchange transformation of \Cref{smooth base change}(1) is an equivalence of $\mathcal{H}^n(S)^\times$-module functors, so [\cite{MR3701352}, Chapter 7, Theorem 3.2.2] again yields the module correspondence presheaf.
\end{proof}

\subsection{Refined pushforward for closed immersions}\label{3.2}
While UGLT closed immersions preserve $n$-birational motivic equivalences by \Cref{generic push}, pulling back along a general closed immersion $i:Z\hookrightarrow S$ can destroy density. In fact, for any closed immersion $i:Z\hookrightarrow X$ of positive codimension, the $0$-dense open subscheme $X\setminus Z\hookrightarrow X$ pulls back to $\emptyset\hookrightarrow Z$ (which is only $(-1)$-dense). This suggests that, instead of asking for the stability of dense open immersions under closed pullback, the right question is what height of density a closed immersion takes to after pullback. Using the supremal codimension $scod(i)$ from \S2.1 (which is finite whenever $S\in \mathrm{Noet} $ is either finite dimensional or catenary, by \Cref{scod finiteness} (1)), we obtain a reasonable shifted pullback-pushforward adjunction:

\begin{thm}\label{closed rel d push}
    Let $i:Z\hookrightarrow S$ be a closed immersion of Noetherian schemes with $scod(i)\leq d$. Then for every $n\geq 0$, $i^*:\mathcal{P}(S)\to \mathcal{P}(Z)$ sends $(n+d)$-birational motivic equivalences to $n$-birational motivic equivalences, and dually $i_*:\mathcal{P}(Z)\to \mathcal{P}(S)$ sends $n$-birational motivic spaces to $(n+d)$-birational motivic spaces, yielding an induced adjunction:
    \[
   i^{n*}=  L^{n}i^*: \mathcal{H}^{n+d}(S) \leftrightarrows \mathcal{H}^{n}(Z):i_{*}.
    \]
\end{thm}
\begin{proof}
    Since $i^*$ preserves colimits and motivic equivalences, it suffices to verify $i^*(B_{n+d}(S))\subset B_n(Z)$. Given $X\in Sm_S$ and $U\hookrightarrow X$ in $B_{n+d}(S)$, consider the pasted pullback squares:
    \[
    \xymatrix{
    (Z\times_S X)\mcap U \ar@{^(->}[r]\ar@{^(->}[d] & Z\times_S X \ar[r]\ar@{^(->}[d]^{i_X} & Z \ar@{^(->}[d]^{i}\\
    U \ar@{^(->}[r] & X \ar[r] & S
    }
    \]
    By \Cref{scod finiteness} (4), $scod(i_X)\leq d$. Every $x\in Z\times_S X$ with $cod_{Z\times_S X}(x)\leq n$ therefore satisfies $cod_X(i_X(x))\leq n+d$ by definition of $scod(i_X)$, hence $i_X(x)\in U$ and $x\in (Z\times_S X)\mcap U$. Thus $(Z\times_S X)\mcap U\hookrightarrow Z\times_S X$ is $n$-dense, and the adjunction follows formally.
\end{proof}

\begin{cor}\label[cor]{closed push fiber}
    Let $S$ be a Noetherian qcqs scheme and $f:\esc{X}\to \esc{Y}$ an $n$-birational motivic equivalence in $\mathcal{P}(S)$. Then for every $s\in S$ with $cod_S(s)\leq n$ (in particular, $s\in S^{(n)}$), the fiber $f_{k(s)}:\esc{X}_{k(s)}\to \esc{Y}_{k(s)}$ is a $0$-birational motivic equivalence in $\mathcal{P}(k(s))$.
\end{cor}
\begin{proof}
    By \Cref{stalks of n bir equiv}, $f_{\mathcal{O}_{S,s}}\in bir_n(\mathcal{O}_{S,s})$. Since $i_s:\mathrm{Spec}(k(s))\hookrightarrow \mathrm{Spec}(\mathcal{O}_{S,s})$ satisfies $scod(i_s)=cod_{\mathcal{O}_{S,s}}(s)=cod_S(s)\leq n$, \Cref{closed rel d push} gives $f_{k(s)}\simeq i_s^*(f_{\mathcal{O}_{S,s}})\in bir_0(k(s))$.
\end{proof}

\begin{thm}\label{closed push full-faithful}
    Let $i:Z\hookrightarrow S$ be a closed immersion in $\mathrm{Noet}$. For every $n\geq 0$ and $d\geq scod(i)$, the functor $i_*:\mathcal{H}^n(Z)\to \mathcal{H}^{n+d}(S)$ induced by \Cref{closed rel d push} is fully faithful.
\end{thm}
\begin{proof}
    The functor ${i_*}_{|\mathcal{H}^n(Z)}$ restricts $i_*^{mot}:\mathcal{H}^{\mathbb{A}^1}(Z)\to \mathcal{H}^{\mathbb{A}^1}(S)$, which is fully faithful by [\cite{khan2016motivic}, Corollary 7.4.3].
\end{proof}

\begin{prop}[Shifted smooth base change for closed immersions]\label[prop]{closed smooth base change}
    For any pullback square in $\mathrm{Noet}$ on the left with $p,p'$ smooth and $i,i'$ closed immersions satisfying $scod(i)\leq d$, we have $scod(i')\leq d$, and for every $n\geq 0$ the exchange transformations
    \begin{flalign}
        p'_{n\shrp}i'^{n*}\to i^{n*}p_{(n+d)\shrp} &\quad\text{from }\mathcal{H}^{n+d}(Y)\text{ to }\mathcal{H}^n(U),\label{csbc1}\\
        p^{(n+d)*}i_*\to i'_*p'^{n*} &\quad\text{from }\mathcal{H}^n(U)\text{ to }\mathcal{H}^{n+d}(Y)\label{csbc2}
    \end{flalign}
    are equivalences, making the shifted squares on the right commute up to canonical equivalence:
    \[
    \xymatrix{
    V\ar@{^(->}[r]^-{i'}\ar[d]_{p'} & Y\ar[d]^{p} && \mathcal{H}^{n+d}(Y) \ar[r]^-{L^ni'^*}\ar[d]_{p_{(n+d)\shrp}} & \mathcal{H}^n(V) \ar[d]^{p'_{n\shrp}} && \mathcal{H}^n(U) \ar[r]^-{i_*}\ar[d]_{p'^{n*}} & \mathcal{H}^{n+d}(X) \ar[d]^{p^{(n+d)*}}\\
    U \ar@{^(->}[r]_-{i} & X && \mathcal{H}^{n+d}(X) \ar[r]_-{L^ni^*} & \mathcal{H}^n(U) && \mathcal{H}^n(V) \ar[r]_-{i'_*} & \mathcal{H}^{n+d}(Y)
    }
    \]
\end{prop}
\begin{proof}
    By \Cref{scod finiteness} (4), $scod(i')\leq d$, and \eqref{csbc2} is the adjoint mate of \eqref{csbc1}. Since $p'_\shrp$ preserves $bir_n$-equivalences (\Cref{smooth extension pull}) and $i^*$ maps $bir_{n+d}(Y)$ into $bir_n(U)$ (\Cref{closed rel d push}), we have $p'_{n\shrp}(L^ni'^*b^{n+d}_Y)\simeq L^n(p'_\shrp i'^* b^{n+d}_Y)$ and $(L^ni^*)p_{(n+d)\shrp}\simeq L^n(i^*p_\shrp b^{n+d}_Y)$, reducing \eqref{csbc1} to smooth base change in $\mathcal{P}(-)$ as in \Cref{smooth base change}.
\end{proof}

\section{Continuity properties of \texorpdfstring{$\mathcal{H}^{n}$}{Hn}}

The Nisnevich sheaf topos presheaf $\mathcal{P}_{nis}(-)^*:\mathrm{Sch}^{op}\to Pr^L$ satisfies (a) Nisnevich descent [\cite{hoyois2017six}, Proposition 4.8(1)] and (b) pro-affine continuity (i.e., its restriction to $\mathrm{Sch}^{aff,op}$ preserves filtered colimits) [\cite{hoyois2015quadratic}, Proposition C.7 (3)]. Both properties descend to the motivic homotopy presheaf $\mathcal{H}^{\mathbb{A}^1}(-)^*:\mathrm{Sch}^{op}\to Pr^L$ ([\cite{hoyois2017six}, Proposition 4.8(2)], [\cite{khan2016motivic}, Proposition 6.1.6], and [\cite{hoyois2015quadratic}, Proposition C.7 (3)]). The presheaf $\mathcal{H}^{\mathbb{A}^1}(-)^*$ satisfies a few more invariances: for example, by construction, it is (c) homotopy invariant (it sends vector bundle projections to fully faithful functors; see \Cref{motivic category is a1 faithful}); whereas, due to the motivic localization sequence, it also satisfies (d) deformation invariance (it sends nil immersions to equivalences). In this section, we establish the analogous properties for $\mathcal{H}^n(-)^*:\mathrm{Sch}^{uglt,op}\to Pr^L$.

\subsection{Nisnevich descent for \texorpdfstring{$\mathcal{H}^{n}(-)^*$}{Hn(-)*}}

\begin{prop}[Nisnevich descent of $\mathcal{H}^n(-)^*$]\label[prop]{bir detect}
    \begin{enumerate}
        \item For a Nisnevich cover $\{f_\alpha: U_\alpha\to S\}$, the pullback functors $f_{\alpha}^*: \mathcal{P}(Sm_S)\to \mathcal{P}(Sm_{U_\alpha})$ preserve and jointly detect $n$-birational motivic equivalences.
        \item The presheaf $\mathcal{H}^{n}(-)^*:\mathrm{Sch}^{uglt,op}\to \widehat{Cat}_\infty$ from \Cref{full functoriality} is Nisnevich separated, i.e., for any Nisnevich cover $\{f_\alpha: U_\alpha\to S\}$, the pullback functors $f_\alpha^{n*}\simeq {f_\alpha^*}_{|\mathcal{H}^n(S)}: \mathcal{H}^{n}(S)\to \mathcal{H}^{n}(U_\alpha)$ are jointly conservative.
        \item The presheaf $\mathcal{H}^{n}(-)^{*}:\mathrm{Sch}^{uglt,op}\longrightarrow Pr^L$ is a Nisnevich sheaf.
    \end{enumerate}
\end{prop}
\begin{proof}
(1). Preservation holds because each $f_\alpha$ is \'etale and hence UGLT (\Cref{generic push}). For the converse, we follow the proof of [\cite{hoyois2017six}, Proposition 4.8(1)]. Since $S$ is quasicompact, we may pass to a finite subcover and set $f:U:=\bigsqcup_\alpha U_\alpha\to S$ with open inclusions $i_\alpha:U_\alpha\hookrightarrow U$. Let $c_\bullet f:\check{C}_\bullet f\to S$ be the \v{C}ech nerve of $f$, and for $[m]\in \Delta$ set $C_m:=(c_m f)_\shrp (c_m f)^*:\mathcal{P}(Sm_S)\to \mathcal{P}(Sm_S)$. For every $F\in \mathcal{P}(Sm_S)$, the augmentation $|\varepsilon_F|:\colim_{[m]\in\Delta^{op}} C_m(F)\to F$ is a Nisnevich equivalence.

    Suppose $h:F\to G$ in $\mathcal{P}(Sm_S)$ satisfies $f_\alpha^*h\in bir_n(U_\alpha)$ for all $\alpha$. Since $f^*h\simeq \bigsqcup_\alpha (i_\alpha)_\shrp f_\alpha^*h$ and each $(i_\alpha)_\shrp$ preserves $bir_n$-equivalences (\Cref{smooth extension pull}), colimit closure gives $f^*h\in bir_n(U)$. For each $m\geq 0$, the $(m+1)$-fold fiber product $c_m f:\check{C}_m f\to S$ factors as $f\circ \phi_m$ for an \'etale morphism $\phi_m:\check{C}_m f\to U$. By \Cref{generic push} and \Cref{smooth extension pull}, $(c_m f)^*h\simeq \phi_m^*(f^*h)\in bir_n(\check{C}_m f)$ and hence $C_m(h)=(c_m f)_\shrp (c_m f)^*h\in bir_n(S)$. By colimit closure of $bir_n(S)$, $\colim C_\bullet(h)\in bir_n(S)$. Finally, in the commutative square
    \[
    \xymatrix{
    \colim C_\bullet(F) \ar[r]^-{|\varepsilon_F|}\ar[d]_{\colim C_\bullet(h)}&F\ar[d]^{h}\\
    \colim C_\bullet(G) \ar[r]_-{|\varepsilon_G|}&G
    }
    \]
    the horizontal arrows lie in $nis(S)\subset bir_n(S)$, so $h\in bir_n(S)$ by two-out-of-three.

    (2). Immediate from (1).

    (3). Since $Pr^L\subset \widehat{Cat}_\infty$ creates limits [\cite{lurie2009higher}, Proposition 5.5.3.13], it suffices to show that for every $S\in \mathrm{Sch}$, the restriction $\mathcal{H}^n(-)^*_{|Et_S}$ is a Nisnevich sheaf of $\infty$-categories. Every morphism $g$ in $Et_S$ is \'etale, so $g^{n*}\simeq {g^*}_{|\mathcal{H}^n}$ by \Cref{smooth extension pull}, making $\mathcal{H}^n(-)^*_{|Et_S}\subset \mathcal{P}_{nis}(-)^*_{|Et_S}$ a full subpresheaf. Since $\mathcal{P}_{nis}(-)^*_{|Et_S}$ is a Nisnevich sheaf [\cite{hoyois2017six}, Proposition 4.8 (1)], it remains to show that if $\{f_\alpha:U_\alpha\to S\}$ is a Nisnevich cover and $\esc{F}\in \mathcal{P}_{nis}(S)$ satisfies $f_\alpha^*\esc{F}\in \mathcal{H}^n(U_\alpha)$ for all $\alpha$, then $\esc{F}\in \mathcal{H}^n(S)$. Since $\{f^*_\alpha: \mathcal{P}_{nis}(S)\to \mathcal{P}_{nis}(U_\alpha)\}$ is jointly conservative [\cite{hoyois2017six}, Proposition 4.8 (1)], for any morphism $V\to X$ in $\mathbb{A}(S)\mcup B_n(S)$, the induced map $\esc{F}^X\to \esc{F}^V$ in $\mathcal{P}_{nis}(S)$ is an equivalence if and only if each $f_\alpha^*(\esc{F}^X\to \esc{F}^V)$ is an equivalence in $\mathcal{P}_{nis}(U_\alpha)$. Under the canonical identification $f_\alpha^*(\esc{F}^W)\simeq (f_\alpha^*\esc{F})^{W\times_S U_\alpha}$ for $W\in Sm_S$, this pullback is the map
    $$(f_\alpha^*\esc{F})^{X\times_S U_\alpha}\longrightarrow (f_\alpha^*\esc{F})^{V\times_S U_\alpha},$$
    which is an equivalence because $V\times_S U_\alpha\to X\times_S U_\alpha$ lies in $\mathbb{A}(U_\alpha)\mcup B_n(U_\alpha)$ (\Cref{ugltpull}) and $f_\alpha^*\esc{F}\in \mathcal{H}^n(U_\alpha)$.
\end{proof}

\begin{cor}[$\mathcal{H}^n(-)^*$ is additive]\label[cor]{zar add}
    For any finite disjoint union of schemes $X=\bigsqcup_\alpha X_\alpha$ (for example, the connected components of $X$ when finite in number) with inclusions $i_\alpha:X_\alpha\hookrightarrow X$, the pullback functor
    $$\prod_\alpha i_\alpha^{n*}:\mathcal{H}^{n}(X)\xrightarrow{\;\sim\;}\prod_\alpha \mathcal{H}^{n}(X_\alpha)$$
    is an equivalence of $\infty$-categories.
\end{cor}
An independent proof of the above corollary follows from the fact that $Sm_X\simeq \prod_\alpha Sm_{X_\alpha} $ and the fact that an $n$-dense open immersion over $X$ is uniquely determined by $n$-dense open immersions over each $X_\alpha$.

\subsection{Deformation invariance and the birational localization sequence}
Recall that the small Nisnevich $\infty$-topos construction is topologically invariant:

\begin{lem}\label[lem]{def inv of small nis topos}
    The presheaf $Shv(-_{nis})^*:\mathrm{Sch}^{op}\longrightarrow Pr^L$, $S\mapsto Shv(S_{nis})$, sends universal homeomorphisms (and, in particular, nil immersions) to equivalences.
\end{lem}
\begin{proof}
This is [\cite{MR1106918}, Proposition E.3]. Note that since radical morphisms are precisely universally injective morphisms [\cite[\href{https://stacks.math.columbia.edu/tag/01S4}{Lemma 01S4}]{stacks-project}], the hypotheses of [\cite{MR1106918}, Proposition E.3] are equivalent to being a universal homeomorphism [\cite[\href{https://stacks.math.columbia.edu/tag/04DF}{Lemma 04DF}]{stacks-project}].
\end{proof}

By contrast, the large Nisnevich $\infty$-topos presheaf $\mathcal{P}_{nis}(-)^*$ is not even nil invariant: for $S=\mathrm{Spec}(k[x]/(x^2))$ and $i:S^{red}=\mathrm{Spec}(k)\hookrightarrow S$, the unit $id\to i_*i^*$ evaluated on $\mathbb{A}^1_S$ is not a Nisnevich equivalence. For example, its value on the terminal object $S$ is the non-invertible map $k[x]/(x^2)=\mathbb{A}^1_S(S)\to \mathbb{A}^1_{S^{red}}(S^{red})=k$. In fact, it turns out that contracting $\mathbb{A}^1$ restores deformation invariance:

\begin{thm}[Motivic deformation invariance]\label{motivic deformation invariant}
    The presheaf $\mathcal{H}^{\mathbb{A}^1}(-)^*: \mathrm{Sch}^{op}\longrightarrow Pr^L$ sends thickenings of schemes to equivalences: for any nil immersion $i:Z\hookrightarrow S$, the adjunction $$L_{mot}i^*:\mathcal{H}^{\mathbb{A}^1}(S)\leftrightarrows \mathcal{H}^{\mathbb{A}^1}(Z):i_{*mot}:={i_*}_{|\mathcal{H}^{\mathbb{A}^1}(Z)}$$ is an equivalence of $\infty$-categories. In particular, $\mathcal{H}^{\mathbb{A}^1}(X)\xrightarrow{\sim} \mathcal{H}^{\mathbb{A}^1}(X^{red})$.
\end{thm}
\begin{proof}
    For any closed immersion $i:Z\hookrightarrow S$ with open complement $j:S\setminus Z\hookrightarrow S$, localization yields a cofiber sequence $\mathcal{H}^{\mathbb{A}^1}(S\setminus Z)\xrightarrow{L_{mot}j_\shrp}\mathcal{H}^{\mathbb{A}^1}(S)\xrightarrow{L_{mot}i^*}\mathcal{H}^{\mathbb{A}^1}(Z)$ in $Pr^L$ [\cite{hoyois2021localization}, Remark 6]. When $i$ is a thickening, $S\setminus Z=\emptyset$ and $\mathcal{H}^{\mathbb{A}^1}(\emptyset)\simeq *$, so $L_{mot}i^*$ (and hence $i_{*mot}$) is an equivalence.
\end{proof}

To deduce deformation invariance for $\mathcal{H}^n(-)^*$, we use the following detection lemma:

\begin{lem}\label[lem]{n bir detection by thickening}
    Let $i:Z \hookrightarrow S$ be a thickening. Then $i_*:\mathcal{P}_{nis}(Z)\to \mathcal{P}_{nis}(S)$ detects $n$-dense locality.
\end{lem}
\begin{proof}
    Suppose $\esc{X}\in \mathcal{P}_{nis}(Z)$ is such that $i_*\esc{X}$ is $B_n(S)$-local, and let $j: U\hookrightarrow X$ be an $n$-dense open immersion in $Sm_Z$. By [\cite{grothendieck1966elements}, Proposition 18.1.1], $X$ admits a Zariski open cover $\{X_\alpha\hookrightarrow X\}$ such that each $X_\alpha \cong X'_\alpha\times_S Z$ for some $X'_\alpha\in Sm_S$. For every tuple $\underline{\alpha}=(\alpha_0,\dots,\alpha_m)$, the intersection $X_{\underline{\alpha}}:=X_{\alpha_0}\mcap \cdots \mcap X_{\alpha_m}$ is an open subscheme of $X_{\alpha_0}$. Since $X_{\alpha_0}\hookrightarrow X'_{\alpha_0}$ is a thickening and hence a homeomorphism, [\cite{grothendieck1966elements}, Th\'eor\`eme 18.1.2] lifts $X_{\underline{\alpha}}\subset X_{\alpha_0}$ and $j_{\underline{\alpha}}:U\mcap X_{\underline{\alpha}}\hookrightarrow X_{\underline{\alpha}}$ to unique open subschemes $X'_{\underline{\alpha}}\subset X'_{\alpha_0}$ and $j'_{\underline{\alpha}}:U'_{\underline{\alpha}}\hookrightarrow X'_{\underline{\alpha}}$ in $B_n(S)$ fitting into pasted cartesian squares:
    \[
    \xymatrix{
    U\mcap X_{\underline{\alpha}} \ar@{^(->}[r]^-{j_{\underline{\alpha}}}\ar@{^(->}[d] & X_{\underline{\alpha}} \ar[r]\ar@{^(->}[d]^{i_{\underline{\alpha}}} & Z \ar@{^(->}[d]^{i}\\
    U'_{\underline{\alpha}} \ar@{^(->}[r]_-{j'_{\underline{\alpha}}} & X'_{\underline{\alpha}} \ar[r] & S
    }
    \]
    i.e., $j_{\underline{\alpha}}=i^*j_{\underline{\alpha}}'$. Since $\esc{X}$ is a (Zariski) sheaf, and $j=\colim_{\underline{\alpha}}^{nis}j_{\underline{\alpha}}$, \v{C}ech descent along $\{X_\alpha\hookrightarrow X\}$ gives
    $$\mathrm{Map}_{\mathcal{P}_{nis}(Z)}(j,\esc{X})\simeq \plim_{[m]\in\Delta}\prod_{\underline{\alpha}}\mathrm{Map}_{\mathcal{P}_{nis}(Z)}(i^*j'_{\underline{\alpha}},\esc{X})\simeq \plim_{[m]\in\Delta}\prod_{\underline{\alpha}}\mathrm{Map}_{\mathcal{P}_{nis}(S)}(j'_{\underline{\alpha}},i_*\esc{X}),$$
    which is an equivalence because each $j'_{\underline{\alpha}}\in B_n(S)$ and $i_*\esc{X}$ is $B_n(S)$-local.
\end{proof}

\begin{thm}[Deformation invariance]\label{deformation invariant}
    The presheaf $\mathcal{H}^{n}(-)^*: \mathrm{Sch}^{uglt,op}\longrightarrow Pr^L$ sends thickenings of schemes to equivalences of $\infty$-categories: if $i:Z\hookrightarrow S$ is a nil immersion, then $i^{n*}:\mathcal{H}^n(S)\to \mathcal{H}^n(Z)$ and its right adjoint $i_*:\mathcal{H}^{n}(Z) \to \mathcal{H}^{n}(S)$ are equivalences of $\infty$-categories. In particular, pullback along the reduction $X^{red}\hookrightarrow X$ induces an equivalence $\mathcal{H}^{n}(X)\xrightarrow{\sim} \mathcal{H}^{n}(X^{red})$.
\end{thm}
\begin{proof}
    Since $i$ is a universal homeomorphism and hence UGLT, \Cref{generic push} yields the adjunction $$i^{n*}:\mathcal{H}^n(S)\leftrightarrows \mathcal{H}^n(Z):i_{n*}\simeq {i_{*mot}}_{|\mathcal{H}^n(Z)}.$$ By \Cref{motivic deformation invariant}, $i_{*mot}$ is an equivalence, so $i_{n*}$ is fully faithful; moreover, any $\esc{X}\in \mathcal{H}^n(S)\subset \mathcal{H}^{\mathbb{A}^1}(S)$ is of the form $i_{*mot}\esc{Y}$ for some $\esc{Y}\in \mathcal{H}^{\mathbb{A}^1}(Z)$, which lies in $\mathcal{H}^n(Z)$ by \Cref{n bir detection by thickening}. Thus $i_{n*}$ (and hence $i^{n*}$) is an equivalence.
\end{proof}

\begin{rem}\label[rem]{Ldense not def inv}
    Unlike $\mathcal{H}^n(-)^*$, the unsheafified localization $L^n_{dense}\mathcal{P}(-)^*$ is neither a Nisnevich sheaf nor deformation invariant, as proof of both facts relies on Zariski descent.
\end{rem}

A similar Zariski-local lifting argument shows that pushforward along any closed immersion of locally Noetherian schemes detects birational locality with a shift by $\mathrm{cod}_S(i)$, yielding a birational localization sequence:

\begin{lem}[Shifted detection of dense locality]\label[lem]{lem:shifted-detection}
    Let $i: Z \hookrightarrow S$ be a closed immersion of Noetherian qcqs schemes with $\mathrm{cod}_S(i) \geq c$, and let $n\geq 0$. If $\esc{X} \in \mathcal{P}_{nis}(Z)$ is such that $i_* \esc{X}$ is $B_{n+c}(S)$-local, then $\esc{X}$ is $B_n(Z)$-local. In particular, if $\esc{X} \in \mathcal{H}^{\mathbb{A}^1}(Z)$ satisfies $i_* \esc{X} \in \mathcal{H}^{n+c}(S)$, then $\esc{X} \in \mathcal{H}^n(Z)$.
\end{lem}
\begin{proof}
    Suppose $\esc{X} \in \mathcal{P}_{nis}(Z)$ is such that $i_* \esc{X}$ is $B_{n+c}(S)$-local, and let $j: U \hookrightarrow X$ be an $n$-dense open immersion in $Sm_Z$. We follow the proof of \Cref{n bir detection by thickening} verbatim up to the construction of the open subschemes $X_{\underline{\alpha}} := X_{\alpha_0} \mcap \dots \mcap X_{\alpha_m} \subset X_{\alpha_0}$ (noting that [\cite{grothendieck1966elements}, Proposition 18.1.1] applies to any closed immersion). Since $i_{\alpha_0}: X_{\alpha_0} \hookrightarrow X'_{\alpha_0}$ is a closed immersion, $X_{\underline{\alpha}} \subset X_{\alpha_0}$ is the pullback $i^*(X'_{\underline{\alpha}})$ of the open subscheme $X'_{\underline{\alpha}} := X'_{\alpha_0} \setminus i_{\alpha_0}(X_{\alpha_0} \setminus X_{\underline{\alpha}}) \subset X'_{\alpha_0}$ in $Sm_S$. Let $i_{\underline{\alpha}}: X_{\underline{\alpha}} \hookrightarrow X'_{\underline{\alpha}}$ be the induced closed immersion, and let $W_{\underline{\alpha}} := X_{\underline{\alpha}} \setminus (U \mcap X_{\underline{\alpha}})$ denote the closed complement of $j_{\underline{\alpha}}: U \mcap X_{\underline{\alpha}} \hookrightarrow X_{\underline{\alpha}}$, which satisfies $\mathrm{cod}_{X_{\underline{\alpha}}}(W_{\underline{\alpha}}) \geq n + 1$ by \Cref{dense prod}. Setting $U'_{\underline{\alpha}} := X'_{\underline{\alpha}} \setminus i_{\underline{\alpha}}(W_{\underline{\alpha}})$ produces pasted cartesian squares exactly as in the proof of \Cref{n bir detection by thickening}, exhibiting $j_{\underline{\alpha}}$ as the pullback $i^*(j'_{\underline{\alpha}})$ of $j'_{\underline{\alpha}}: U'_{\underline{\alpha}} \hookrightarrow X'_{\underline{\alpha}}$ in $Sm_S$. Moreover, for every point $x \in W_{\underline{\alpha}}$, say lying over $z \in Z$, comparing the flat dimension formula \Cref{scod finiteness}(3) for the smooth morphisms $X'_{\underline{\alpha}} \to S$ and $X_{\underline{\alpha}} \to Z$ yields
    \[
    \mathrm{cod}_{X'_{\underline{\alpha}}}(i_{\underline{\alpha}}(x)) = \mathrm{cod}_{X_{\underline{\alpha}}}(x) + \mathrm{cod}_S(i(z)) - \mathrm{cod}_Z(z)\geq (n+1)+cod_S(i) \geq (n + 1) + c = n + c + 1,
    \]
    whence $\mathrm{cod}_{X'_{\underline{\alpha}}}(i_{\underline{\alpha}}(W_{\underline{\alpha}}))\geq n+c+1$ so that its complement $j'_{\underline{\alpha}} \in B_{n+c}(S)$. The remainder of the proof now proceeds verbatim as in the final sentence of \Cref{n bir detection by thickening}, with $B_n(S)$ replaced by $B_{n+c}(S)$.
\end{proof}

As noted in \S1, because $L^n$ is not locally cartesian ([\cite{0bat}, Counterexample 3.2.15]), one cannot directly use the standard unstable formalism (as in [\cite{hoyois2017six}, Theorem 4.12, Corollary 4.13, Proposition 5.1]) to show that $i_*:\mathcal{P}_{nis}(Z)_\bullet\to \mathcal{P}_{nis}(S)_\bullet$ preserves birational motivic equivalences so as to admit a right adjoint (say up to a height shift). Using \Cref{lem:shifted-detection}, we obtain this adjunction and its associated fiber sequence from the motivic localization sequence instead:

\begin{prop}[Shifted exceptional pullback for closed immersions]\label[prop]{excp push for closed}
    Let $i: Z \hookrightarrow X$ be a closed immersion of codimension $c:=\mathrm{cod}(i)$ with open complement $j: U \hookrightarrow X$, and assume that $X$ is Noetherian. Let $n \geq 0$.
    \begin{enumerate}
        \item The motivic exceptional pullback $i^!:\mathcal{H}^{\mathbb{A}^1}_\bullet(X)\to \mathcal{H}^{\mathbb{A}^1}_\bullet(Z)$ restricts to a functor $i^!_n:\mathcal{H}^n_\bullet(X)\to \mathcal{H}^{n-c}_\bullet(Z)$. Dually, $i_*:\mathcal{H}^{\mathbb{A}^1}_\bullet(Z)\to \mathcal{H}^{\mathbb{A}^1}_\bullet(X)$ sends $(n-c)$-birational motivic equivalences to $n$-birational motivic equivalences, inducing an adjunction in $Pr^L_{\bullet}$:
        \[
        i_{n*}:=L^n i_*: \mathcal{H}^{n-c}_\bullet(Z)\leftrightarrows \mathcal{H}^n_\bullet(X): i^!_n.
        \]
     Moreover, whenever $\mathrm{scod}(i)=c$ we have $L^ni_*\simeq i_*L^{n-c}$; in particular $i_{n*}\simeq {i_*}_{|\mathcal{H}^{n-c}_\bullet(Z)}$.
        \item For every $\esc{X} \in \mathcal{H}^n_\bullet(X)$, there is a canonical fiber sequence in $\mathcal{H}^n_\bullet(X)$:
        \[
        i_* i^!_n\esc{X} \longrightarrow \esc{X} \longrightarrow j_* j^{n*}\esc{X}.
        \]
    \end{enumerate}
\end{prop}
\begin{proof}
    Let $\esc{X}\in \mathcal{H}^n_\bullet(X)$. By \Cref{generic push} and \Cref{smooth extension pull}, $j_* j^{n*}\esc{X}\simeq j_* j^*\esc{X}$ lies in $\mathcal{H}^n_\bullet(X)$. Since $\mathcal{H}^n_\bullet(X)\subset \mathcal{H}^{\mathbb{A}^1}_\bullet(X)$ is closed under limits, the fiber $\mathcal{F}:=\mathrm{fib}(\esc{X}\to j_* j^{n*}\esc{X})$ computed in $\mathcal{H}^{\mathbb{A}^1}_\bullet(X)$ lies in $\mathcal{H}^n_\bullet(X)$ and satisfies $\mathcal{F}\simeq i_* i^!\esc{X}$ by the motivic localization sequence [\cite{hoyois2017six}, Proposition 5.2] establishing the fiber sequence of (2) once $i^!\esc{X}\in \mathcal{H}^{n-c}_\bullet(Z)$ is verified.

    When $n<c$, $j:U\hookrightarrow X$ is $n$-dense, so $\esc{X}\xrightarrow{\sim} j_*j^{n*}\esc{X}$ by \Cref{n dense equiv}, and $i^!\esc{X}\simeq *_Z$, which of course belongs to $\mathcal{H}^{n-c}_\bullet(Z)$ ($\simeq *$). When $n\geq c$, since $i_*(i^!\esc{X})\simeq \mathcal{F}\in \mathcal{H}^n_\bullet(X)$, we obtain $i^!\esc{X}\in \mathcal{H}^{n-c}_\bullet(Z)$ by \Cref{lem:shifted-detection}. Thus $i^!$ restricts to $i^!_n:\mathcal{H}^n_\bullet(X)\to \mathcal{H}^{n-c}_\bullet(Z)$. By adjunction ($i_*\dashv i^!$ on $\mathcal{H}^{\mathbb{A}^1}_\bullet$), $i_*$ sends $(n-c)$-birational motivic equivalences to $n$-birational motivic equivalences, inducing $i_{n*}:=L^n i_*\dashv i^!_n$ in $Pr^L_{\bullet}$. Finally, when $\mathrm{scod}(i)=c$, $i_{n*}\simeq i_*$ follows from \Cref{closed rel d push} for $n\geq c$ (and trivially on $\mathcal{H}^{n-c}_\bullet(Z)\simeq *$ for $n<c$), completing (1) and (2).
\end{proof}

\begin{thm}[Birational localization sequence]\label{thm:localization-sequence}
    Let $i: Z \hookrightarrow X$ be an equicodimensional closed immersion of Noetherian schemes of codimension $d = \mathrm{cod}(i) = \mathrm{scod}(i)$ (so $\delta(i) = 0$) with open complement $j: U \hookrightarrow X$, and let $n \geq d$. Then:
    \begin{enumerate}
        \item For every $\esc{E} \in \mathcal{H}^n(X)$, there is a canonical equivalence $L^n(i_* L_{mot} i^*\esc{E}) \xrightarrow{\sim} i_* L^{n-d} i^*\esc{E}$ and a cocartesian square in $\mathcal{H}^n(X)$:
        \[
        \xymatrix{
        j_{n\shrp} j^{n*}\esc{E} \ar[r] \ar[d] &\esc{E} \ar[d] \\
        j_{n\shrp}(*_U) \ar[r] & i_* L^{n-d} i^*\esc{E}
        }
        \]
        which for $\esc{E} \in \mathcal{H}^n_\bullet(X)$ yields a cofiber sequence in $\mathcal{H}^n_\bullet(X)$ (also valid for $0 \leq n < d$):
        \[
        j_{n\shrp} j^{n*}\esc{E} \longrightarrow\esc{E} \longrightarrow i_* L^{n-d} i^*\esc{E}.
        \]
        \item The sequence
        \[
        \mathcal{H}^{n-d}(Z) \xrightarrow{i_*} \mathcal{H}^n(X) \xrightarrow{j^{n*}} \mathcal{H}^n(U)
        \]
        is a fiber sequence in $\widehat{Cat}_\infty$ and $Pr^R$ (over $*_U$). In fact, the sequence
        \[
        \mathcal{H}^n(U) \xrightarrow{j_{n\shrp}} \mathcal{H}^n(X) \xrightarrow{L^{n-d} i^*} \mathcal{H}^{n-d}(Z)
        \]
        of its left adjoints is a fiber-cofiber sequence in $Pr^L$.
    \end{enumerate}
\end{thm}
\begin{proof}
    Since $i_*: \mathcal{H}^{n-d}(Z)\to \mathcal{H}^{n}(X)$ is fully faithful (\Cref{closed push full-faithful}) and $j^{n*} i_* \simeq *_U$, for the first sequence in (2) it suffices to note that any $\mathcal{Q}\in \mathcal{H}^n(X)$ with $j^{n*}\mathcal{Q}\simeq *_U$ satisfies $\mathcal{Q}\simeq i_*\esc{Y}$ for a unique $\esc{Y}\in \mathcal{H}^{\mathbb{A}^1}(Z)$ by motivic localization [\cite{hoyois2021localization}, Remark 6], and $\esc{Y}\in \mathcal{H}^{n-d}(Z)$ by \Cref{lem:shifted-detection}.

    Applying the left adjoint $L^n: \mathcal{H}^{\mathbb{A}^1}(X) \to \mathcal{H}^n(X)$ to the motivic fracture pushout square [\cite{morel19991}, Theorem 3.2.21] yields a cocartesian square in $\mathcal{H}^n(X)$ with bottom-right vertex $\mathcal{Q} := L^n(i_* L_{mot} i^*\esc{E})$. Applying $j^{n*}$ and using $j^{n*} j_{n\shrp} \simeq \mathrm{id}$ (\Cref{fully faithful j_*}) gives $j^{n*}\mathcal{Q} \simeq *_U$, so $\mathcal{Q} \simeq i_* \esc{Y}$ for a unique $\esc{Y} \in \mathcal{H}^{n-d}(Z)$ by the above. Since $i_*(\mathcal{H}^{n-d}(Z)) \subset \mathcal{H}^n(X)$ (\Cref{closed rel d push}), full faithfulness of $i_*$ yields for every $\mathcal{V} \in \mathcal{H}^{n-d}(Z)$:
    \begin{flalign*}
        \mathrm{Map}_{\mathcal{H}^{n-d}(Z)}(\esc{Y}, \mathcal{V}) \simeq \mathrm{Map}_{\mathcal{H}^n(X)}(L^n i_* L_{mot} i^*\esc{E}, i_* \mathcal{V}) &\simeq \mathrm{Map}_{\mathcal{H}^{\mathbb{A}^1}(Z)}(L_{mot} i^*\esc{E}, \mathcal{V})\\
        &\simeq \mathrm{Map}_{\mathcal{H}^{n-d}(Z)}(L^{n-d} i^*\esc{E}, \mathcal{V}),
    \end{flalign*}
    whence $\esc{Y} \simeq L^{n-d} i^*\esc{E}$ by Yoneda, proving the unpointed part of (1). The pointed cofiber sequence (for all $n\geq 0$) follows directly by applying $L^n$ to the pointed motivic localization sequence [\cite{hoyois2017six}, Proposition 5.2] and using \Cref{excp push for closed}(1).

    By $Pr^L\simeq (Pr^R)^{op}$, the second sequence in (2) is a cofiber sequence in $Pr^L$. It remains to show that it is also a fiber sequence in $Pr^L$ (i.e., in $\widehat{Cat}_\infty$ over $\emptyset_Z \in \mathcal{H}^{n-d}(Z)$). Since $j_{n\shrp}$ is fully faithful (\Cref{fully faithful j_*}) and $L^{n-d}i^* j_{n\shrp} \simeq L^{n-d}i^* L^n j_\shrp \simeq L^{n-d}i^* j_\shrp \simeq \emptyset_Z$ (\Cref{closed rel d push}), it suffices to show that any $\esc{X} \in \mathcal{H}^n(X)$ with $L^{n-d}i^*\esc{X} \simeq \emptyset_Z$ lies in the essential image of $j_{n\shrp}$. Indeed, since $i_*\emptyset_Z \simeq h_X(U)$, the unit $\esc{X} \to i_*L^{n-d}i^*\esc{X} \simeq h_X(U)$ shows that $\esc{X}$ lies in $\mathcal{P}(X)_{/h_X(U)}$, which forces $j_\shrp^{\mathrm{pre}} j^*\esc{X} \xrightarrow{\sim} \esc{X}$ in $\mathcal{P}(Sm_X)$, and hence $j_{n\shrp} j^{n*}\esc{X} \xrightarrow{\sim} \esc{X}$.
\end{proof}
\begin{prop}[Closed projection formula]\label[prop]{closed proj formula}
    Let $i:Z\hookrightarrow S$ be an equicodimensional closed immersion of codimension $d$ with $S$ Noetherian. Then $i_*:\mathcal{H}^{n-d}_\bullet(Z)\to \mathcal{H}^n_\bullet(S)$ lifts to a morphism in $\mathrm{Mod}_{\mathcal{H}^n_\bullet(S)^{\wedge^n}}(Pr^{L,\otimes})$: for all $\esc{A},\esc{C}\in \mathcal{H}^n_\bullet(S)$ and $\esc{B}\in \mathcal{H}^{n-d}_\bullet(Z)$, the canonical morphisms
    \[
    \esc{A}\wedge^n i_*\esc{B}\longrightarrow i_*\big(L^{n-d}i^*\esc{A}\wedge^{n-d}\esc{B}\big), \qquad (i^!_n\esc{C})_\bullet^{L^{n-d}i^*\esc{A}}\longrightarrow i^!_n\big(\esc{C}_\bullet^{\esc{A}}\big)
    \]
    are equivalences in $\mathcal{H}^n_\bullet(S)$ and $\mathcal{H}^{n-d}_\bullet(Z)$, respectively.
\end{prop}
\begin{proof}
    The two morphisms are adjoint mates under $\wedge^n\dashv (-)_\bullet^{(-)}$ (\Cref{Ln is cart} (2)) and $i_*\dashv i^!_n$ (\Cref{excp push for closed}), so it suffices to check the first, which also yields the module lift via [\cite{lurie2017higher}, Corollary 7.3.2.7]. By \Cref{closed rel d push} and \Cref{Ln is cart}, $L^{n-d}i^*$ is smash monoidal, and the lax-linear map factors as
    \[
    \esc{A}\wedge^n i_*\esc{B} \xrightarrow{\;\eta\;} i_* L^{n-d}i^*(\esc{A}\wedge^n i_*\esc{B}) \xrightarrow{\;\sim\;} i_*\big(L^{n-d}i^*\esc{A}\wedge^{n-d} L^{n-d}i^* i_*\esc{B}\big) \xrightarrow{\;\sim\;} i_*\big(L^{n-d}i^*\esc{A}\wedge^{n-d}\esc{B}\big),
    \]
    where the last map is an equivalence since $L^{n-d}i^* i_*\simeq id$ (\Cref{closed push full-faithful}). By the localization cofiber sequence (\Cref{thm:localization-sequence} (1)), $\eta$ is the cofiber of $j_{n\shrp} j^{n*}(\esc{A}\wedge^n i_*\esc{B})\to \esc{A}\wedge^n i_*\esc{B}$, whose source vanishes since $j^{n*}$ (for the open complement $j:U\hookrightarrow S$) is monoidal so that
    \[
    j_{n\shrp} j^{n*}(\esc{A}\wedge^n i_*\esc{B}) \simeq j_{n\shrp}\big(j^{n*}\esc{A}\wedge^n j^{n*}i_*\esc{B}\big) \simeq j_{n\shrp}\big(j^{n*}\esc{A}\wedge^n *_U\big) \simeq *_S.
    \]
    Hence $\eta$ is an equivalence.
\end{proof}
\subsection{Essentially smooth continuity for \texorpdfstring{$\mathcal{H}^n$}{Hn}}
Recall first that $\mathcal{P}(-)^*$, $\mathcal{P}_{nis}(-)^*$, and $\mathcal{H}^{\mathbb{A}^1}(-)^*:\mathrm{Sch}^{op}\to Pr^L$ are pro-affine continuous:

\begin{prop}[\textup{[\cite{hoyois2015quadratic}, Proposition C.7(3)]}]\label[prop]{pro-affine continuity}
    Let $S=\plim_\alpha S_\alpha$ be a cofiltered limit of qcqs schemes with affine transition morphisms and projections $\phi_\alpha:S\to S_\alpha$. Then the canonical assembly functors
    \begin{flalign*}
        \dcolim_\alpha \phi_\alpha^*:\dcolim_\alpha \mathcal{P}(S_\alpha)\xrightarrow{\;\sim\;}\mathcal{P}(S), \quad \dcolim_\alpha L_{nis}\phi_\alpha^*:\dcolim_\alpha \mathcal{P}_{nis}(S_\alpha)\xrightarrow{\;\sim\;}\mathcal{P}_{nis}(S), \\ \dcolim_\alpha L_{mot}\phi_\alpha^*:\dcolim_\alpha \mathcal{H}^{\mathbb{A}^1}(S_\alpha)\xrightarrow{\;\sim\;}\mathcal{H}^{\mathbb{A}^1}(S)
    \end{flalign*}
    are equivalences in $Pr^L$; equivalently, the right adjoints $\plim_\alpha \phi_{\alpha,*}:\mathcal{P}(S)\to \plim_\alpha \mathcal{P}(S_\alpha)$ (and their restrictions to $\mathcal{P}_{nis}$ and $\mathcal{H}^{\mathbb{A}^1}$) are equivalences in $\widehat{Cat}_\infty$ (and $Pr^R$).
\end{prop}
\begin{proof}
    Since colimits in $Pr^L$ correspond to limits in $Pr^R\subset \widehat{Cat}_\infty$ [\cite{lurie2009higher}, Proposition 5.5.3.13 and Theorem 5.5.3.18], both formulations are equivalent. By finite-presentation descent of schemes and morphisms [\cite[\href{https://stacks.math.columbia.edu/tag/01ZM}{Lemma 01ZM}, \href{https://stacks.math.columbia.edu/tag/01ZB}{Section 01ZB}]{stacks-project}], together with descent of smoothness and separatedness [\cite[\href{https://stacks.math.columbia.edu/tag/0C0C}{Lemma 0C0C}, \href{https://stacks.math.columbia.edu/tag/01ZQ}{Lemma 01ZQ}]{stacks-project}], the pullback functors induce an equivalence $\dcolim_\alpha Sm_{S_\alpha}\xrightarrow{\sim} Sm_S$, whence $\mathcal{P}(S)=\mathrm{Fun}(Sm_S^{op},Spc)\simeq \plim_\alpha \mathcal{P}(S_\alpha)$. Descent of \'etale morphisms [\cite[\href{https://stacks.math.columbia.edu/tag/07RP}{Lemma 07RP}]{stacks-project}], open immersions [\cite[\href{https://stacks.math.columbia.edu/tag/0EUU}{Lemma 0EUU}]{stacks-project}], surjectivity [\cite[\href{https://stacks.math.columbia.edu/tag/07RR}{Lemma 07RR}]{stacks-project}], and isomorphisms [\cite[\href{https://stacks.math.columbia.edu/tag/081E}{Lemma 081E}]{stacks-project}] lifts Nisnevich distinguished squares from $Sm_S$ to some $Sm_{S_\alpha}$, while $\mathbb{A}^1$-projections lift trivially, so the equivalence restricts to $\mathcal{P}_{nis}$ and $\mathcal{H}^{\mathbb{A}^1}$.
\end{proof}

In the birational setting, arbitrary affine transition maps need not lift generalizations, so pro-affine continuity requires flatness. But if $S=\plim_\alpha S_\alpha$ is a cofiltered limit in $\mathrm{Sch}$ with flat affine transition maps $\phi_{\beta\alpha}:S_\beta\to S_\alpha$, applying the methods discussed in the paragraph above \Cref{stalks of n bir equiv} to each $S_\alpha$, every projection $\phi_\alpha:S\to S_\alpha$ is flat, so $\phi_\alpha^{n*}:\mathcal{H}^n(S_\alpha)\to \mathcal{H}^n(S)$ is defined in $Pr^L$ (\Cref{flat pull push}).

\begin{thm}[Flat-affine continuity of $\mathcal{H}^n(-)^*$]\label{continuity of Hn}
    Let $S=\plim_\alpha S_\alpha$ be a cofiltered limit of qcqs schemes with flat affine transition morphisms and projections $\phi_\alpha:S\to S_\alpha$. Assume either that:
    \begin{enumerate}
        \item $n=0$ and all transition morphisms are dominant; or
        \item $S$ and each $S_\alpha$ are Noetherian (so the diagram lies in $\mathrm{Noet}^{flat}$).
    \end{enumerate}
    Then the canonical assembly functor in $Pr^L$,
    $$\phi=\dcolim_{\alpha} \phi_\alpha^{n*}: \dcolim_{\alpha}\mathcal{H}^n(S_\alpha)\xrightarrow{\;\sim\;} \mathcal{H}^n(S),$$
    is an equivalence of presentable $\infty$-categories; equivalently, $\plim_{\alpha} \phi_{\alpha,n*}: \mathcal{H}^n(S)\xrightarrow{\;\sim\;} \plim_{\alpha}\mathcal{H}^n(S_\alpha)$ is an equivalence in $\widehat{Cat}_\infty$ (and $Pr^R$).
\end{thm}
\begin{proof}
    Given $\mathcal{H}^{\mathbb{A}^1}(S)\simeq \plim_\alpha \mathcal{H}^{\mathbb{A}^1}(S_\alpha)$ (\Cref{pro-affine continuity}), it suffices to show that every $j:U\hookrightarrow X$ in $B_n(S)$ is pulled back from some morphism in $B_n(S_\alpha)$. By \Cref{pro-affine continuity} and [\cite[\href{https://stacks.math.columbia.edu/tag/01Z4}{Lemma 01Z4}]{stacks-project}], there exist $\alpha_0$ and a quasicompact open immersion $j_{\alpha_0}:U_{\alpha_0}\hookrightarrow X_{\alpha_0}$ in $Sm_{S_{\alpha_0}}$ pulling back to $j$. Denote the induced flat projection $X\to X_{\alpha_0}$ again by $\phi_{\alpha_0}$.

    Under (1), for any nonempty affine open $V\subset S_{\alpha_0}$, dominance of the affine transition maps implies that each $\phi_{\beta\alpha_0}^{-1}(V)$ is nonempty, hence so is $\phi_{\alpha_0}^{-1}(V)=\mathrm{Spec}(\dcolim_{\beta}\mathcal{O}(\phi_{\beta\alpha_0}^{-1}(V)))$; thus $\phi_{\alpha_0}:S\to S_{\alpha_0}$ is dominant. Since $X_{\alpha_0}\to S_{\alpha_0}$ is smooth (hence universally open), the base change $\phi_{\alpha_0}:X\to X_{\alpha_0}$ is also dominant. Consequently $\phi_{\alpha_0}\circ j:U\to X_{\alpha_0}$ is dominant, which forces $j_{\alpha_0}\in B_0(S_{\alpha_0})$ and proves (1).

    Under (2), let $Z:=X\setminus j(U)$ and let $Z_0:=\overline{\phi_{\alpha_0}(Z)}\subset X_{\alpha_0}$ be the closure of its image under the flat morphism $\phi_{\alpha_0}:X\to X_{\alpha_0}$. Since $Z_0\subset X_{\alpha_0}\setminus U_{\alpha_0}$ and $\phi_{\alpha_0}^{-1}(U_{\alpha_0})=U=X\setminus Z$, we have $\phi_{\alpha_0}^{-1}(Z_0)=Z$. Setting $U'_{\alpha_0}:=X_{\alpha_0}\setminus Z_0\supset U_{\alpha_0}$ gives complementary cartesian squares over $\phi_{\alpha_0}:S\to S_{\alpha_0}$:
    \[
    \xymatrix{
    U \ar@{^(->}[r]^-{j}\ar[d] & X \ar[d]^{\phi_{\alpha_0}} & Z \ar@{_(->}[l]\ar[d]\\
    U'_{\alpha_0} \ar@{^(->}[r]_-{j'_{\alpha_0}} & X_{\alpha_0} & Z_0 \ar@{_(->}[l]
    }
    \]
    so the open immersion $j'_{\alpha_0}:U'_{\alpha_0}\hookrightarrow X_{\alpha_0}$ (which lies in $Sm_{S_{\alpha_0}}$ as $X_{\alpha_0}$ is Noetherian) pulls back to $j$. Moreover, since $Z$ is Noetherian, $Z_0$ is the union of $\overline{\{\phi_{\alpha_0}(\eta)\}}$ as $\eta$ ranges over the finite set $Z^{(0)}$, so every generic point $\xi\in Z_0^{(0)}$ is of the form $\phi_{\alpha_0}(\eta)$ for some $\eta\in Z^{(0)}$. Since $\phi_{\alpha_0}^{-1}(\xi)\subset \phi_{\alpha_0}^{-1}(Z_0)=Z$ and $\eta\in Z^{(0)}$, $\eta$ is a generic point of the fiber $\phi_{\alpha_0}^{-1}(\xi)$, so $cod_{\phi_{\alpha_0}^{-1}(\xi)}(\eta)=0$. The flat dimension formula \Cref{scod finiteness}(3) applied to $\phi_{\alpha_0}$ at $\eta$ then yields
    $$cod_{X_{\alpha_0}}(\xi)=cod_X(\eta)\geq cod_X(Z)\geq n+1.$$
    Hence $cod_{X_{\alpha_0}}(Z_0)\geq n+1$, so $j'_{\alpha_0}\in B_n(S_{\alpha_0})$, proving (2).
\end{proof}
We now proceed to our main result, the generic decomposition of the birational motivic homotopy category. We first isolate the remaining technical input:
\begin{lem}\label[lem]{irred open}
    Every quasi-separated scheme $X$ with finitely many irreducible components admits a dense open subscheme $U$ such that $U\to X$ is quasicompact and each connected component of $U$ is irreducible.
\end{lem}
\begin{proof}
    If $\{X_i\}_{i\in I}$ is the finite set of irreducible components of $X$, then $Z:=\bigcup_{i\neq j}(X_i\mcap X_j)$ contains no generic points of $X$, so $U:=X\setminus Z=\bigsqcup_{i\in I}(X_i\setminus Z)$ is a dense open subscheme with irreducible connected components; replacing each component by a dense affine open subscheme ensures $U\to X$ is quasicompact since $X$ is quasi-separated [\cite[\href{https://stacks.math.columbia.edu/tag/03GI}{Lemma 03GI}]{stacks-project}].
\end{proof}

\begin{cor}[Generic decomposition of $\mathcal{H}^0(-)$]\label[cor]{generic decomposition}
    Let $X$ be a qcqs scheme with finitely many generic points $X^{(0)}$. Then pullback along the pro-smooth morphism $\bigsqcup_{\eta\in X^{(0)}}\mathrm{Spec}(k(\eta))\hookrightarrow X$ induces an equivalence of $\infty$-categories:
    $$\mathcal{H}^0(X)\xrightarrow{\;\sim\;}\prod_{\eta\in X^{(0)}}\mathcal{H}^0(k(\eta)).$$
    In particular, for any reduced irreducible qcqs scheme $X$ (such as a variety $V$), $\mathcal{H}^0(X)\simeq \mathcal{H}^0(K(X))$.
\end{cor}
\begin{proof}
    Combining \Cref{irred open}, dense open invariance (\Cref{dense equiv}), deformation invariance (\Cref{deformation invariant}), and additivity (\Cref{zar add}), we may reduce to the case where $X=\mathrm{Spec}(A)$ is an integral affine scheme with function field $K(X)$. Then $\mathrm{Spec}(K(X))=\plim_{f\in A\setminus \{0\}}\mathrm{Spec}(A_f)$ is a cofiltered limit of dense principal affine open subschemes of $X$ along affine open immersions (which are thus flat and dominant). Applying \Cref{continuity of Hn}(1) yields $\mathcal{H}^0(K(X))\simeq \plim_{f}\mathcal{H}^0(\mathrm{Spec}(A_f))$. But due to \Cref{dense equiv}, for each $f$, $\mathcal{H}^0(\mathrm{Spec}(A_f))\simeq \mathcal{H}^0(\mathrm{Spec}(A))=\mathcal{H}^0(X)$ so that the limit identifies with $\mathcal{H}^0(X)$ itself.
\end{proof}

\begin{cor}\label{bir topos on fin gen pt}
    Let $S$ be a qcqs scheme with finitely many generic points. Then $\mathcal{H}^{0}(S)$ is an $\infty$-topos.
\end{cor}
\begin{proof}
    By \Cref{generic decomposition}, $\mathcal{H}^0(S)\simeq \prod_{\eta\in S^{(0)}}\mathcal{H}^0(k(\eta))$, where each factor $\mathcal{H}^0(k(\eta))\simeq \mathcal{H}^b(k(\eta))$ (\Cref{Motivic equivalences are Birational equivalences}) is an $\infty$-topos by [\cite{0bat}, Theorem 3.1.13] and a product of $\infty$-topoi is an $\infty$-topos [\cite{lurie2009higher}, Proposition 6.3.2.1].
\end{proof}

\begin{cor}\label[cor]{fiberwise criteria}
    Let $S$ be a qcqs scheme and $f:\esc{X}\to \esc{Y}$ a morphism in $\mathcal{P}(S)$. If $f$ is a $0$-birational motivic equivalence, then for every $\eta\in S^{(0)}$, $f_\eta:\esc{X}_{k(\eta)}\to \esc{Y}_{k(\eta)}$ is a $0$-birational motivic equivalence in $\mathcal{P}(k(\eta))$. Conversely, if $S$ has finitely many generic points, then $f\in bir_0(S)$ if and only if $f_\eta\in bir_0(k(\eta))$ for every $\eta\in S^{(0)}$.
\end{cor}
\begin{proof}
    The first is \Cref{stalks of n bir equiv} at $\eta\in S^{(0)}$, and when $S$ has finitely many generic points, the converse follows from \Cref{generic decomposition}.
\end{proof}

\textit{Birational motivic equivalence of schemes.}

\begin{cor}\label[cor]{fiberwise criteria for smoo}
    Let $S$ be a qcqs scheme with finitely many generic points and $f:X\to Y$ a smooth morphism in $Sm_S$. If for every generic point $\eta\in Y^{(0)}$, the generic fiber $X_{k(\eta)}$ is contractible in $\mathcal{H}^0(k(\eta))$, then $f$ is a $0$-birational motivic equivalence in $\mathcal{P}(S)$.
\end{cor}
\begin{proof}
    Since $q:Y\to S$ is smooth and hence quasicompact by assumption, $Y$ also has finitely many generic points. By \Cref{fiberwise criteria}, $f:X\to Y=*_Y$ lies in $bir_0(Y)$, so $f=q_\shrp(f)\in bir_0(S)$ by \Cref{smooth extension pull}.
\end{proof}
\begin{rem}
    A similar version of the Noetherian stratification method of [\cite{A1contractibledubouloz}, Proposition 1.15] (which detects ordinary motivic equivalences fiberwise over all points of a finite-dimensional base) does not adapt to $\mathcal{H}^0$ as $i_*$ does not preserve $0$-birationally contractible spaces under pushforward (\S\ref{3.2}), as is required in the last line of the proof of [\cite{A1contractibledubouloz}, Proposition 1.15]. Also, in contrast to \textit{loc.\ cit.}, the other (non-generic) fibers of a birational equivalence need not be birational equivalences. Moreover, smoothness of $f$ is indispensable: for instance, $f:\mathbb{A}^1_k\msqcup \mathrm{Spec}(k)\to \mathbb{A}^1_k$ has contractible generic fiber with a birationally disconnected source but a birationally contractible codomain.
\end{rem}

\begin{cor}\label[cor]{detect bir equiv of smooth by bir contr fibers}
    Let $S$ be a qcqs scheme with finitely many generic points (for example, $S=\mathrm{Spec}(k)$ for a field $k$), and let $f:X\to Y$ be a smooth morphism in $Sm_S$ whose fiber $X_{k(\eta)}$ over every generic point $\eta\in Y^{(0)}$ is retract-rational over $k(\eta)$. Then $f$ is a $0$-birational motivic equivalence in $\mathcal{P}(S)$.
\end{cor}
\begin{proof}
    Every retract-rational smooth $k(\eta)$-scheme is a retract in $\mathcal{H}^0(k(\eta))$ of a rational scheme and is therefore contractible in $\mathcal{H}^0(k(\eta))$, so \Cref{fiberwise criteria for smoo} applies.
\end{proof}

\begin{rem}
    When the generic fibers $X_{k(\eta)}$ are (stably) rational, spreading out over the finite set $Y^{(0)}$ shows that $f$ is already a stably birational isomorphism (\Cref{stable bir is bir equiv}), whereas retract-rational generic fibers need not make $f$ stably birational.
\end{rem}
\begin{ex}\label[ex]{stable bir is bir equiv}
    Every birational morphism, and more generally every stably birational morphism, in $Sm_S$ is a $0$-birational motivic equivalence (since $\mathbb{P}^n_S\to S$ is a birational equivalence and $L^0$ preserves finite products by \Cref{Ln is cart}). By contrast, dense open immersions of singular $S$-schemes need not induce birational equivalences in $\mathcal{P}(S)$ [\cite{0bat}, Counterexample 3.1.8]. Every scdh equivalence is a $0$-birational equivalence [\cite{0bat}, Remark 2.1.5].
\end{ex}

\begin{cor}[cdh-local birational contractibility]\label[cor]{cdh local bir contr}
    Let $|S^{(0)}|<\infty$, $f:X\to Y$ a smooth morphism in $Sm_S$, and $\mathcal{U}=\{U_\alpha\to Y\}$ a cdh cover with $X_\alpha:=X\times_Y U_\alpha$. If $(X_\alpha)_{k(u)}$ is contractible in $\mathcal{H}^0(k(u))$ for every $\alpha$ and every $u\in U_\alpha$ lying over $Y^{(0)}$, then $f\in bir_0(S)$. When $\mathcal{U}$ is generating, it suffices to test $u\in U_\alpha^{(0)}$ (in which case $h^0_{U_\alpha}(X_\alpha)\simeq *_{U_\alpha}$).
\end{cor}
\begin{proof}
    By \Cref{fiberwise criteria for smoo}, it suffices to check $X_{k(\eta)}\simeq *_{k(\eta)}$ in $\mathcal{H}^0(k(\eta))$ for every $\eta\in Y^{(0)}$. Complete decomposition lifts $\eta$ to some $u\in U_\alpha$ with $k(\eta)\xrightarrow{\sim} k(u)$, yielding $X_{k(\eta)}\cong (X_\alpha)_{k(u)}\simeq *_{k(\eta)}$. When $\mathcal{U}$ is generating (a finite composition of Nisnevich covers and abstract blowup covers), induction on the cover tree yields such a lift $u\in U_\alpha^{(0)}$: Nisnevich covers have discrete fibers, so any lift of a generic point is generic; in an abstract blowup $\{Z\hookrightarrow T, p:W\to T\}$, every $\tau\in T^{(0)}$ lies in $Z^{(0)}$ if $\tau\in Z$, and otherwise lies in $T\setminus Z\cong W\setminus p^{-1}(Z)$ and lifts to $W^{(0)}$. Finally, $h^0_{U_\alpha}(X_\alpha)\simeq *_{U_\alpha}$ is equivalent to $(X_\alpha)_u\simeq *$ $\forall u\in U_\alpha^{(0)}$ by \Cref{fiberwise criteria for smoo}.
\end{proof}

\begin{rem}\label[rem]{rem:cdh-local-sharpness}
   (a) Without assuming $|S^{(0)}|<\infty$, when $\mathcal{U}$ is a generating scdh cover in $Sm_S$ with $h^0_{U_\alpha}(X_\alpha)\simeq *_{U_\alpha}$ for all $\alpha$, $f\in bir_0(S)$ still follows from $scdh(S)\subset bir_0(S)$; by contrast, the restricted \v{C}ech nerve of a general cdh cover need not augment to a $0$-birational equivalence in $\mathcal{P}(S)$ (e.g., $S^{red}\hookrightarrow S$ for non-reduced $S$).

   (b) Testing only $u\in U_\alpha^{(0)}$ fails for arbitrary cdh covers. Indeed, $U:=\mathrm{Spec}(\mathbb{R}[x,y]/(x^2+y^2))$ is a cdh cover of $\mathrm{Spec}(\mathbb{R})$, with dense open $U\setminus\{0\}\cong \mathbb{G}_{m,\mathbb{C}}$ (via $i:=y/x$). For $X:=\mathrm{Spec}(\mathbb{R}[u,v]/(u^2+v^2+1))\in Sm_{\mathbb{R}}$, the splitting $X_{\mathbb{C}}\cong \mathbb{G}_{m,\mathbb{C}}$ yields $h^0_U(X_U)\simeq *_U$, whereas $h^0_{\mathbb{R}}(X)\not\simeq *_{\mathbb{R}}$, separated by the sheaf $G\in \mathcal{H}^0(\mathbb{R})$ sending $W\in Sm_{\mathbb{R}}$ to $\emptyset$ if $W(\mathbb{R})\neq\emptyset$ and $*$ otherwise (in other words, since $\mathbb{R}$ is ample and $X$ does not have an $\mathbb{R}$-point; [\cite{0bat}, Proposition 3.4.8]). Thus \Cref{cdh local bir contr} reflects generic-point lifting rather than categorical cdh descent for $\mathcal{H}^0(-)^*$ (which is not even defined on $\mathrm{Sch}$).
\end{rem}
\begin{prop}\label[prop]{bir rigid curves}
    Over a field $k$, any $0$-birational motivic equivalence between $0$-dimensional smooth $k$-schemes, or between smooth proper $k$-curves of positive genus, is an isomorphism. By contrast, over an algebraically closed field $k$, any morphism between connected smooth proper $k$-curves of genus $0$ is a $0$-birational motivic equivalence.
\end{prop}
\begin{proof}
    Zero-dimensional smooth $k$-schemes and smooth proper $k$-curves of positive genus (see, e.g., [\cite{levine2010slices}, p.~399]) are birationally rigid and hence represent discrete objects of $\mathcal{H}^0(k)$, so any equivalence between them in $\mathcal{H}^0(k)$ is an isomorphism of schemes. Over an algebraically closed field $k$, every connected smooth proper curve of genus $0$ is isomorphic to $\mathbb{P}^1_k$, which is contractible in $\mathcal{H}^0(k)$.
\end{proof}
\textit{\'Etale birational equivalences.} We next characterize dense open immersions homotopically as the \'etale birational equivalences. Let us first recall the geometric characterization:

\begin{lem}\label[lem]{geom bir etale is dense open}
    A separated \'etale morphism of schemes $f:X\to Y$ is a birational isomorphism if and only if $f$ is a dense open immersion.
\end{lem}
\begin{proof}
    The `if' direction is immediate by definition. Conversely, since $f$ is separated and \'etale, the diagonal $\Delta_f:X\to X\times_Y X$ is an open and closed immersion [\cite[\href{https://stacks.math.columbia.edu/tag/02GE}{Lemma 02GE}]{stacks-project}]. Let $Y\xleftarrow{b} U\xrightarrow{a} X$ be the birational isomorphism ($fa=b$). Now, consider the commutative square:
    \[
    \xymatrix{
    U \ar@{^(->}[r]\ar[d]_{\Delta_{f|_U}}^{\wr} & X \ar[d]^{\Delta_f}\\
    U\times_Y U \ar@{^(->}[r] & X\times_Y X
    }
    \]
    By \Cref{dense prod}, $U\times_Y U$ is a dense open subscheme of $X\times_Y X$. Since $U\to Y$ is a monomorphism, $\Delta_{f|_U}=\Delta_b$ is an isomorphism. Thus $U\times_Y U=\Delta_f(U)\subset \Delta_f(X)$, so the clopen subscheme $\Delta_f(X)\subset X\times_Y X$ is dense and hence an equality. Hence $\Delta_f$ is an isomorphism, meaning $f$ is an \'etale monomorphism and therefore an open immersion [\cite[\href{https://stacks.math.columbia.edu/tag/025G}{Theorem 025G}]{stacks-project}], which is dense since $U\subset  X$ is dense.
\end{proof}
We now describe the homotopical characterization of dense open immersions among \'etale maps:
\begin{cor}\label[cor]{dense open into base}\label[cor]{n dense open into base}
    Let $S$ be a qcqs scheme, $f:X\to Y$ a morphism in $Sm_S$, and $n\geq 0$, assuming $S$ is Noetherian when $n\geq 1$. Then the following are equivalent:
    \begin{enumerate}
        \item $f$ is an $n$-dense open immersion (i.e., a dense open immersion when $n=0$).
        \item $f$ is \'etale and an $n$-birational motivic equivalence in $\mathcal{P}(Y)$ (so when $Y=S$, $f\in bir_n(S)$).
        \item $f$ is \'etale and for every $y\in Y$ with $cod_Y(y)\leq n$ (i.e., $y\in Y^{(0)}$ when $n=0$), the fiber $X_{k(y)}$ is contractible in $\mathcal{H}^0(k(y))$.
    \end{enumerate}
\end{cor}
\begin{proof}
    Replacing $S$ by $Y$ (which is qcqs, and Noetherian when $n\geq 1$), $(1)\implies (2)$ holds by definition, and $(2)\implies (3)$ follows from \Cref{fiberwise criteria} when $n=0$ and \Cref{closed push fiber} when $n\geq 1$. For $(3)\implies (1)$, we may assume that $Y$ is reduced. Now, each generic fiber $f_\eta:X_{k(\eta)}\to \mathrm{Spec}(k(\eta))$ ($\eta\in Y^{(0)}$) is a $0$-dimensional smooth $k(\eta)$-scheme that is contractible in $\mathcal{H}^0(k(\eta))$ and hence an isomorphism by \Cref{bir rigid curves}. Since $f$ is finitely presented, [\cite[\href{https://stacks.math.columbia.edu/tag/081E}{Lemma 081E}]{stacks-project}] provides for each $\eta\in Y^{(0)}$ an open neighborhood $V_\eta\subset Y$ over which $f_{V_\eta}:f^{-1}(V_\eta)\xrightarrow{\sim} V_\eta$ is an isomorphism. Setting $V:=\bigcup_{\eta\in Y^{(0)}}V_\eta$, $V$ is dense open in $Y$, $U:=f^{-1}(V)$ is dense open in $X$ by \Cref{glt pull}, and $f|_U:U\xrightarrow{\sim} V$ is an isomorphism, making $f$ a birational isomorphism. Since $f$ is \'etale, it is a dense open immersion by \Cref{geom bir etale is dense open}. Finally, when $n\geq 1$, for every $y\in Y$ with $cod_Y(y)\leq n$, contractibility of $X_{k(y)}$ in $\mathcal{H}^0(k(y))$ forces $X_{k(y)}\neq \emptyset$ (since $h^0_{k(y)}(\emptyset)\not\simeq *_{k(y)}$), whence $y\in f(X)$ and $f$ is $n$-dense.
\end{proof}

\subsection{Rationality property of \texorpdfstring{$\mathcal{H}^0(-)$}{H0(-)}}

\begin{thm}\label{H^b is bir local}
    The presheaf $\mathcal{H}^{0}(-)^{*}:\mathrm{Sch}^{uglt,op}\to Pr^L$ is a birational presheaf, i.e.,
    $$\mathcal{H}^{0}(-)^*\in L_{dense}\mathcal{P}_{\Sigma}(\mathrm{Sch}^{uglt}, Pr^L).$$
\end{thm}
\begin{proof}
    Additivity follows from \Cref{zar add}, and invariance under $0$-dense open immersions follows from \Cref{dense equiv}.
\end{proof}

\begin{rem}
    Conversely, combining \Cref{dense equiv} and \Cref{zar add} with the saturation-class proof of \Cref{nis is bir} gives an independent proof of Nisnevich descent (\Cref{bir detect} (3)) for $\mathcal{H}^0(-)^*:\mathrm{Sch}^{uglt,op}\to Pr^L$.
\end{rem}

\begin{lem}\label[lem]{nis local ff}
    Let $q:X\to S$ be an UGLT morphism of schemes and $n\geq 0$. If there exists a Nisnevich cover $\{f_\alpha:U_\alpha\to S\}$ such that each base change $q_\alpha^{n*}:\mathcal{H}^n(U_\alpha)\to \mathcal{H}^n(X\times_S U_\alpha)$ is fully faithful, then $q^{n*}:\mathcal{H}^n(S)\to \mathcal{H}^n(X)$ is fully faithful.
\end{lem}
\begin{proof}
    The functor $q^{n*}$ is fully faithful if and only if the unit $\eta:id_{\mathcal{H}^n(S)}\to q_{n*}q^{n*}$ is an equivalence. Since each $f_\alpha$ is \'etale (hence smooth) and $q$ is UGLT, smooth base change (\Cref{smooth base change} (2)) identifies $f_\alpha^{n*}(\eta)$ with $\eta_\alpha ({f_\alpha^{n*}}_X)$, where $\eta_\alpha:id_{\mathcal{H}^n(U_\alpha)}\xrightarrow{\sim} q_{\alpha,n*}q_\alpha^{n*}$ is the unit of $q_\alpha^{n*}\dashv q_{\alpha,n*}$. Nisnevich conservativity (\Cref{bir detect} (2)) concludes the proof.
\end{proof}

\begin{prop}[Smooth full faithfulness criterion]\label[prop]{sm ff crit}
    Let $q:X\to S$ be a smooth morphism of schemes and $n\geq 0$. Then $q^{n*}:\mathcal{H}^n(S)\to \mathcal{H}^n(X)$ is fully faithful if and only if $h^n_S(X)\simeq *_S$ in $\mathcal{H}^n(S)$. In particular, when $n=0$ and $|S^{(0)}|<\infty$, $q^{0*}$ is fully faithful if and only if $X_{k(\eta)}$ is contractible in $\mathcal{H}^0(k(\eta))$ for every $\eta\in S^{(0)}$.
\end{prop}
\begin{proof}
    Since $q$ is smooth, $q^{n*}$ is fully faithful if and only if the counit $\varepsilon:q_{n\shrp}q^{n*}\to id_{\mathcal{H}^n(S)}$ is an equivalence. Evaluating at $*_S$ gives $h^n_S(X)\to *_S$, proving necessity. Conversely, as $q_{n\shrp}q^{n*}$ preserves colimits, it suffices to test $\varepsilon$ on the generators $h^n_S(T)$ ($T\in Sm_S$, \Cref{compact generation}), where $\varepsilon_{h^n_S(T)}$ is $h^n_S(X\times_S T\to T)\simeq (h^n_S(X)\to *_S)\times id_{h^n_S(T)}$ by \Cref{bir cart}. The final assertion is \Cref{fiberwise criteria}.
\end{proof}

\begin{cor}\label[cor]{cdh ff}
    Let $q:X\to Y$ be a smooth morphism satisfying the hypotheses of \Cref{cdh local bir contr} or \Cref{rem:cdh-local-sharpness} (a) (with $S=Y$). Then $q^{0*}:\mathcal{H}^0(Y)\to \mathcal{H}^0(X)$ is fully faithful.
\end{cor}
\begin{proof}
    Combine \Cref{sm ff crit} with \Cref{cdh local bir contr} or \Cref{rem:cdh-local-sharpness} (a).
\end{proof}

Relativizing the notion of retract rational schemes, we say that a smooth $T$-scheme $Z\in Sm_T$ is \emph{retract rational} over $T$ if there is a $0$-dense open subscheme of $Z$ that is a $T$-retract of a $0$-dense open subscheme of $\mathbb{A}^m_T$ for some $m\geq 0$.

\begin{cor}[Homotopy and rational invariance]\label[cor]{homotopy}
    Let $q:X\to S$ be a morphism of schemes. If $q$ is a Nisnevich-locally trivial $\mathbb{A}^m$-bundle (in particular, a vector or affine bundle), then $q^{n*}:\mathcal{H}^n(S)\to \mathcal{H}^n(X)$ is fully faithful for every $n\geq 0$. For $n=0$, the same holds for any smooth bundle trivialized by retract rational schemes in $Sm_S$ over a generating scdh cover in $Sm_S$, or over an arbitrary cdh cover when $|S^{(0)}|<\infty$.
\end{cor}
\begin{proof}
    For the first claim, \Cref{nis local ff} reduces to $X=\mathbb{A}^m_S$, where $h^n_S(\mathbb{A}^m_S)\simeq *_S$ and \Cref{sm ff crit} applies. The second is a special case of \Cref{cdh ff}.
\end{proof}

\begin{cor}[Pure transcendence invariance]\label[cor]{trans}
    Let $f:X\to S$ be a morphism of schemes with $|S^{(0)}|<\infty$ inducing a bijection $X^{(0)}\xrightarrow{\sim} S^{(0)}$ such that each $k(f(\xi))\subset k(\xi)$ is purely transcendental of finite transcendence degree. Then the induced generic pullback $\mathcal{H}^0(S)\to \mathcal{H}^0(X)$ (given by $f^{0*}$ whenever $f$ is UGLT) is fully faithful. In particular, $\mathcal{H}^0(k)\to \mathcal{H}^0(k(t_1,\dots,t_n))$ is fully faithful for any field $k$.
\end{cor}
\begin{proof}
    For a field $k$, \Cref{generic decomposition} identifies $\mathcal{H}^0(k(t_1,\dots,t_n))\simeq \mathcal{H}^0(\mathbb{A}^n_k)$, reducing to \Cref{homotopy}. In general, \Cref{generic decomposition} reduces $\mathcal{H}^0(S)\to \mathcal{H}^0(X)$ to the finite product $\prod_{\eta\in S^{(0)}}\mathcal{H}^0(k(\eta))\to \prod_{\xi\in X^{(0)}}\mathcal{H}^0(k(\xi))$ over $X^{(0)}\xrightarrow{\sim} S^{(0)}$, which is fully faithful by the field case.
\end{proof}

\begin{rem}\label[rem]{rem about strong a1 inv}
    In view of \Cref{H^b is bir local}, \Cref{bir local is P1 local}, and \Cref{homotopy}, one might ask whether $\mathcal{H}^0(-)^*$ inverts $\mathbb{A}^1$-projections, making it a motivic $\infty$-category in the sense of [\cite{kato2017motivic}]. However, the blowup proof of \Cref{bir local is P1 local} ([\cite{MR3404383}, Appendix A]; [\cite{cisinski2023homotopy}, Proposition 1.3]) requires pullbacks along non-UGLT closed immersions; indeed, \Cref{bir local is P1 local} fails on $Sm_k^{uglt}$, as the $0$-dense local presheaf $X\mapsto \prod_{\eta\in X^{(0)}}k(\eta)$ sends $\mathbb{A}^1_k\to \mathrm{Spec}(k)$ to $k\hookrightarrow k(t)$. Moreover, the following counterexample shows that $p^{0*}:\mathcal{H}^0(k)\hookrightarrow \mathcal{H}^0(\mathbb{A}^1_k)\simeq \mathcal{H}^0(k(t))$ is never an equivalence over any field $k$.
\end{rem}

\begin{coex}
    For any field $k$ and the generic point $p:\mathrm{Spec}(k(t))\to \mathrm{Spec}(k)$ of $\mathbb{A}^1_k$, the fully faithful embedding $p^{0*}:\mathcal{H}^0(k)\hookrightarrow \mathcal{H}^0(k(t))\simeq \mathcal{H}^0(\mathbb{A}^1_k)$ is not essentially surjective. Choose a prime $\ell$ invertible in $k$ and set $L:=\mathrm{Spec}(k(t)[x]/(x^\ell-t))\in Sm_{k(t)}$. Being smooth of dimension $0$, $h_{k(t)}(L)$ lies in $\mathcal{H}^0(k(t))$ (\Cref{bir rigid curves}), and its image $p_{0*}h_{k(t)}(L)\in \mathcal{H}^0(k)$ is initial: for any nonempty affine $X=\mathrm{Spec}(A)\in Sm_k$, a $\bar{k}$-point $A\to \bar{k}$ induces $A\otimes_k k(t)\to \bar{k}(t)$, so any element of $(p_{0*}h_{k(t)}(L))(X)\simeq \mathrm{Hom}_{Sm_{k(t)}}(X_{k(t)},L)$ would give $y\in \bar{k}(t)$ with $y^\ell=t$, which is absurd. Thus $p^{0*}p_{0*}h_{k(t)}(L)\simeq \emptyset_{k(t)}\not\simeq h_{k(t)}(L)$, so $p^{0*}$ is not an equivalence.
\end{coex}
\subsection{Unshifted functoriality and continuity at height \texorpdfstring{$\infty$}{infinity}}\label{sec:height-infty}

Let $S$ be a qcqs scheme, and let
$$\mathcal{H}^{\infty}(S):=\plim_{n\in\mathbb{N}}^{Pr^L}\mathcal{H}^n(S)\simeq \dcolim_{n\in\mathbb{N}}^{Pr^R}\mathcal{H}^n(S)$$
denote the limit of the unstable birational tower (\S\ref{bst}). Since limits in $Pr^L$ are computed in $\widehat{Cat}_\infty$ [\cite{lurie2009higher}, Proposition 5.5.3.13], an object of $\mathcal{H}^\infty(S)$ is a compatible tower $\esc{X}_\bullet=\{\esc{X}_n\}_{n\geq 0}$ with $\esc{X}_n\in \mathcal{H}^n(S)$ and $L^n\esc{X}_{n+1}\xrightarrow{\sim} \esc{X}_n$. A morphism in $\mathcal{P}(S)$ is inverted by the canonical left adjoint $L^\infty: \mathcal{P}(S)\to \mathcal{H}^{\mathbb{A}^1}(S)\to \mathcal{H}^\infty(S)$, $\esc{X}\mapsto \{L^n\esc{X}\}_{n\geq 0}$, if and only if it lies in $bir_\infty(S) := \bigcap_{n\geq 0} bir_n(S)$. $bir_\infty(S)$ is clearly saturated, whence by [\cite{lurie2009higher}, 5.5.4.15], the subcategory $L_{bir_\infty(S)}\mathcal{H}^{\mathbb{A}^1}(S)\subset \mathcal{H}^{\mathbb{A}^1}(S)$ of $bir_\infty(S)$-local motivic spaces is an accessible localization fitting into the chain
$$\mathcal{I}\simeq \mathcal{H}^{-1}(S)\subset \mathcal{H}^0(S)\subset \mathcal{H}^1(S)\subset \cdots \subset L_{bir_\infty(S)}\mathcal{H}^{\mathbb{A}^1}(S)\subset \mathcal{H}^{\mathbb{A}^1}(S),$$
and $L^\infty:\mathcal{H}^{\mathbb{A}^1}(S)\to \mathcal{H}^\infty(S)$ factors in $Pr^L$ over each finite-height localization $L^n$ into an accessible localization followed by a conservative left adjoint:
\[
\xymatrix{
\mathcal{H}^{\mathbb{A}^1}(S) \ar[r]^-{L_{bir_\infty}}\ar[dr]_-{L^n} & L_{bir_\infty(S)}\mathcal{H}^{\mathbb{A}^1}(S) \ar[r]^-{\Phi}\ar[d] & \mathcal{H}^\infty(S) \ar[dl]^-{\mathrm{ev}_n}\\
& \mathcal{H}^n(S) &
}
\]
where $\Phi(\esc{X})=\{L^i\esc{X}\}_{i\geq 0}$ has right adjoint $R^\infty(\{\esc{X}_i\}_{i\geq 0})=\plim_i^{\mathcal{H}^{\mathbb{A}^1}(S)}\esc{X}_i$.

\begin{rem}[Convergence of the Birational tower]\label[rem]{2.2.3}
    It is unclear whether either functor in $L^\infty \simeq \Phi\circ L_{bir_\infty}$ is an equivalence: $L_{bir_\infty}$ is an equivalence if and only if $bir_\infty(S)=mot(S)$, whereas $\Phi$ is an equivalence if and only if for every compatible tower $\{\esc{X}_i\}_{i\geq 0}\in \mathcal{H}^\infty(S)$, the limit $\esc{X}:=\plim_i^{\mathcal{H}^{\mathbb{A}^1}(S)}\esc{X}_i$ satisfies $L^k\esc{X}\xrightarrow{\sim}\esc{X}_k$ for all $k\geq 0$.
\end{rem}

Passing to the limit over $n\in\mathbb{N}$ in $Pr^L$ (and in $\mathrm{CAlg}(Pr^{L,\otimes})$ via [\cite{lurie2017higher}, Corollary 3.2.2.5]) and intersecting the properties of $bir_n$ over $n\geq 0$, \Cref{full functoriality}, \Cref{bir detect}, \Cref{deformation invariant}, and \Cref{homotopy} immediately yield:

\begin{prop}\label[prop]{infty birational primary functoriality}
    \begin{enumerate}
        \item For every qcqs scheme $S$, $\mathcal{H}^\infty(S)$ and $L_{bir_\infty(S)}\mathcal{H}^{\mathbb{A}^1}(S)$ carry canonical cartesian structures in $\mathrm{CAlg}(Pr^{L,\otimes})$ such that $L_{bir_\infty}$ and $\Phi$ are symmetric monoidal.
        \item The assignments $S\mapsto \mathcal{H}^\infty(S)$ and $S\mapsto L_{bir_\infty(S)}\mathcal{H}^{\mathbb{A}^1}(S)$ extend to presheaves
        $$\mathcal{H}^\infty(-),\; L_{bir_\infty(-)}\mathcal{H}^{\mathbb{A}^1}(-): {\mathrm{Cor}(\mathrm{Sch})_{uglt,sm}}^{op}\longrightarrow Pr^L$$
        whose horizontal restrictions $\mathcal{H}^\infty(-)^*$ and $L_{bir_\infty(-)}\mathcal{H}^{\mathbb{A}^1}(-)^*$ on $\mathrm{Sch}^{uglt,op}$ lift to $\mathrm{CAlg}(Pr^{L,\otimes})$, satisfy Nisnevich descent, send nil immersions to equivalences, and send Nisnevich-locally trivial $\mathbb{A}^m$-bundles to fully faithful functors.
    \end{enumerate}
\end{prop}

\begin{rem}
    Unlike the finite-height case (\Cref{continuity of Hn}), flat-affine continuity on $\mathrm{Noet}^{flat,op}$ does not automatically pass to height $\infty$: in the proof of \Cref{continuity of Hn}, the index $\alpha_0$ at which an $n$-dense open immersion descends depends on $n$, and filtered colimits in $Pr^L$ need not commute with countable limits.
\end{rem}

Let $\mathrm{N}$ denote either of the full subcategories of $\mathrm{Noet}$ consisting of finite-dimensional schemes or universally catenary schemes (note that $\mathrm{N}$ is closed under finite-type morphisms). Over $\mathrm{N}$, the finite-height obstructions to general finite-type pullback functoriality, namely the height shift by $scod$ (\Cref{closed rel d push}), vanish at height $\infty$:

\begin{lem}[Bounded height shift along finite-type morphisms]\label[lem]{ft bounded shift}
    Let $f:X\to Y$ be a finite-type morphism in $\mathrm{N}$. There exists an integer $d\geq 0$ (with $d=scod(f)$ when $f$ is a closed immersion) such that for every smooth base-change square
    \[
    \xymatrix{
    X\times_Y Y' \ar[r]^-{p'}\ar[d]_{f'} & X \ar[d]^{f}\\
    Y' \ar[r]_-{p} & Y
    }
    \]
    with $p$ smooth, and for every $n\geq 0$, $f'^*$ maps $bir_{n+d}(Y')$ into $bir_n(X\times_Y Y')$, inducing
    $$L^nf'^*:\mathcal{H}^{n+d}(Y')\leftrightarrows \mathcal{H}^n(X\times_Y Y'):f'_*$$
    and a shifted smooth base change equivalence $p'_{n\shrp}(L^nf'^*)\xrightarrow{\sim} (L^nf^*)p_{(n+d)\shrp}$.
\end{lem}
\begin{proof}
    When $f$ is a closed immersion, $d:=scod(f)<\infty$ (\Cref{scod finiteness} (1)) satisfies the claims directly by \Cref{closed rel d push} and \Cref{closed smooth base change}. In general (assuming $X\neq \emptyset$), choose a finite affine open cover $\{v_\beta:V_\beta\hookrightarrow Y\}$ and finite affine open covers $\{U_{\alpha\beta}\hookrightarrow f^{-1}(V_\beta)\}_\alpha$ (with composite open immersions $u_{\alpha\beta}:U_{\alpha\beta}\hookrightarrow X$) such that each $U_{\alpha\beta}\to V_\beta$ factors through a closed immersion $i_{\alpha\beta}$ into affine space over $V_\beta$:
    \[
    \xymatrix{
    U_{\alpha\beta} \ar@{^(->}[r]^-{i_{\alpha\beta}}\ar@{^(->}[d]_{u_{\alpha\beta}} & \mathbb{A}^{N_{\alpha\beta}}_{V_\beta} \ar[r]^-{q_{\alpha\beta}} & V_\beta \ar@{^(->}[d]^{v_\beta}\\
    X \ar[rr]_-{f} && Y
    }
    \]
    Since each $\mathbb{A}^{N_{\alpha\beta}}_{V_\beta}\in \mathrm{N}$, we have $d:=\max_{\alpha,\beta}\{scod(i_{\alpha\beta})\}<\infty$ by \Cref{scod finiteness} (1).

    Given any smooth morphism $p:Y'\to Y$ with base change $f':X':=X\times_Y Y'\to Y'$ (automatically in $\mathrm{N}$), pulling back the covers and factorizations along $p$ yields a finite Zariski open cover $\{u'_{\alpha\beta}:U'_{\alpha\beta}\hookrightarrow X'\}$, open immersions $v'_\beta:V'_\beta\hookrightarrow Y'$, and factorizations
    $$f'\circ u'_{\alpha\beta} = v'_\beta\circ q'_{\alpha\beta}\circ i'_{\alpha\beta} : U'_{\alpha\beta} \xhookrightarrow{i'_{\alpha\beta}} \mathbb{A}^{N_{\alpha\beta}}_{V'_\beta} \xrightarrow{q'_{\alpha\beta}} V'_\beta \xhookrightarrow{v'_\beta} Y'.$$
    Because $i'_{\alpha\beta}$ is the base change of $i_{\alpha\beta}$ along the smooth (hence flat) projection $\mathbb{A}^{N_{\alpha\beta}}_{V'_\beta}\to \mathbb{A}^{N_{\alpha\beta}}_{V_\beta}$ in $\mathrm{N}$, \Cref{scod finiteness} (4) gives $scod(i'_{\alpha\beta})\leq scod(i_{\alpha\beta})\leq d$. For any $h\in bir_{n+d}(Y')$, smoothness of $v'_\beta$ and $q'_{\alpha\beta}$ implies $q'^*_{\alpha\beta}v'^*_\beta(h)\in bir_{n+d}(\mathbb{A}^{N_{\alpha\beta}}_{V'_\beta})$ by \Cref{generic push}, whence
    $$u'^*_{\alpha\beta}(f'^*h)\simeq i'^*_{\alpha\beta}q'^*_{\alpha\beta}v'^*_\beta(h)\in bir_n(U'_{\alpha\beta})$$
    by \Cref{closed rel d push}. Since $\{u'_{\alpha\beta}:U'_{\alpha\beta}\hookrightarrow X'\}$ is a Zariski cover, \Cref{bir detect} (1) implies $f'^*h\in bir_n(X')$, inducing the adjunction $L^nf'^*:\mathcal{H}^{n+d}(Y')\leftrightarrows \mathcal{H}^n(X'):f'_*$.

    Finally, since $p'_\shrp$ preserves $bir_n$-equivalences (\Cref{smooth extension pull}) and $f^*$ maps $bir_{n+d}(Y)$ into $bir_n(X)$ (taking $p=id_Y$ above), we have $p'_{n\shrp}(L^nf'^*)\simeq L^n(p'_\shrp f'^* b^{n+d}_{Y'})$ and $(L^nf^*)p_{(n+d)\shrp}\simeq L^n(f^*p_\shrp b^{n+d}_{Y'})$. The exchange transformation $p'_{n\shrp}(L^nf'^*)\to (L^nf^*)p_{(n+d)\shrp}$ is therefore $L^n$ applied to the presheaf smooth base change equivalence $p'_\shrp f'^*\xrightarrow{\sim} f^*p_\shrp:\mathcal{P}(Y')\to \mathcal{P}(X)$, just as in \Cref{smooth base change}.
\end{proof}

\begin{thm}[Unshifted correspondence functoriality at height $\infty$]\label{infty birational functoriality}
    \begin{enumerate}
        \item For any finite-type morphism $f:X\to Y$ in $\mathrm{N}$ and any $d\geq 0$ as in \Cref{ft bounded shift}, $f^*$ preserves $bir_\infty$-equivalences and the shifted adjunctions $L^nf^*:\mathcal{H}^{n+d}(Y)\leftrightarrows \mathcal{H}^n(X):f_*$ assemble over $n\in\mathbb{N}$ to canonical adjunctions in $Pr^L$ (independent of $d$):
        $$L_{bir_\infty}f^*: L_{bir_\infty(Y)}\mathcal{H}^{\mathbb{A}^1}(Y)\leftrightarrows L_{bir_\infty(X)}\mathcal{H}^{\mathbb{A}^1}(X): f_*, \qquad f^{\infty *}:\mathcal{H}^\infty(Y)\leftrightarrows \mathcal{H}^\infty(X):f_{\infty *},$$
        where $\big(f^{\infty *}(\esc{X}_\bullet)\big)_n\simeq L^n f^*(\esc{X}_{n+d})$.
        \item Over $\mathrm{N}$, the assignments $S\mapsto \mathcal{H}^\infty(S)$ and $S\mapsto L_{bir_\infty(S)}\mathcal{H}^{\mathbb{A}^1}(S)$ extend to presheaves (agreeing with \Cref{infty birational primary functoriality} (2) on finite-type UGLT and smooth morphisms)
        $$\mathcal{H}^\infty(-),\; L_{bir_\infty(-)}\mathcal{H}^{\mathbb{A}^1}(-): {\mathrm{Cor}(\mathrm{N})_{ft,sm}}^{op}\longrightarrow Pr^L$$
        whose restrictions to $\mathrm{N}^{ft,op}$ are Nisnevich sheaves lifting to $\mathrm{CAlg}(Pr^{L,\otimes})$. In particular, for any $S\in \mathrm{N}$, they lift on finite-type $S$-schemes $\mathrm{FT}_S$ to module presheaves over $\mathrm{Cor}(\mathrm{FT}_S)_{all,sm}^{op}$.
    \end{enumerate}
\end{thm}
\begin{proof}
    (1). By \Cref{ft bounded shift}, $f^*$ maps $bir_\infty(Y)=\bigcap_{n\geq 0}bir_{n+d}(Y)$ into $\bigcap_{n\geq 0}bir_n(X)=bir_\infty(X)$, giving the first adjunction. Cofinality of $n\mapsto n+d$ in $\mathbb{N}$ then yields $f^{\infty *}:\mathcal{H}^\infty(Y)\simeq \plim_n^{Pr^L}\mathcal{H}^{n+d}(Y)\to \mathcal{H}^\infty(X)$ with right adjoint $f_{\infty *}$. For any $d'\geq d$ as in \Cref{ft bounded shift}, $f^*$ sends $\esc{X}_{n+d'}\to \esc{X}_{n+d}\in bir_{n+d}(Y)$ into $bir_n(X)$, so $L^nf^*(\esc{X}_{n+d'})\xrightarrow{\sim} L^nf^*(\esc{X}_{n+d})$ and $f^{\infty *}$ is independent of $d$.

    (2). By (1) and \Cref{smooth extension pull}, finite-type pullbacks and smooth extensions preserve $bir_\infty$-equivalences. The correspondence functor for $\mathcal{H}^{\mathbb{A}^1}(-)$ therefore descends to $L_{bir_\infty(-)}\mathcal{H}^{\mathbb{A}^1}(-)$. Since $bir_\infty$ is the intersection of the classes $bir_n$, the localization $L_{bir_\infty}$ preserves finite products by \Cref{bir cart}, hence so does each $L_{bir_\infty}f^*$, giving the lift to $\mathrm{CAlg}(Pr^{L,\otimes})$ [\cite{lurie2017higher}, Corollary 2.4.1.8].

    For $\mathcal{H}^\infty(-)$, we first construct the finite-type pullbacks coherently. Working in $\mathrm{Ind}(Pr^R)$ (in a larger universe), consider the filtered colimit
    \[
        \mathcal{H}^{\mathrm{ind}}(S)
        := \dcolim_n^{\mathrm{Ind}(Pr^R)} \mathcal{H}^n(S).
    \]
    Since the inclusions $\mathcal{H}^n(S)\hookrightarrow\mathcal{H}^{\mathbb{A}^1}(S)$ are fully faithful right adjoints [\cite{lurie2009higher}, Proposition 5.5.4.20], evaluation at $T\in Pr^R$ identifies $\mathcal{H}^{\mathrm{ind}}(S)(T)$ with the union of those components of $\mathrm{Map}_{Pr^R}(T,\mathcal{H}^{\mathbb{A}^1}(S))$ consisting of right adjoints that factor through some $\mathcal{H}^n(S)$. For a finite-type map $f:X\to Y$, \Cref{ft bounded shift} gives $f_*\bigl(\mathcal{H}^n(X)\bigr)\subseteq \mathcal{H}^{n+d}(Y)$ in $Pr^R$. Thus postcomposition with $f_*$ preserves these unions of components, as does precomposition in $T$. They therefore define a subfunctor of the ambient mapping-space functor (the Yoneda image of $\mathcal{H}^{\mathbb{A}^1}(-)_*$), with all composition coherences inherited from the ambient pushforwards, yielding a functor $\mathcal{H}^{\mathrm{ind}}(-)_*:\mathrm{N}^{ft}\to\mathrm{Ind}(Pr^R)$.

    Taking the filtered colimit in $Pr^R$ and passing to left adjoints gives $\mathcal{H}^\infty(-)^*:\mathrm{N}^{ft,op}\to Pr^L$, recovering $f^{\infty *}$ from (1) with $\mathrm{ev}_n\circ f^{\infty *}\simeq L^n f^*\circ\mathrm{ev}_{n+d}$. On finite-type UGLT maps we may take $d=0$, so this subfunctor coincides with the diagramwise colimit $\dcolim_n^{\mathrm{Ind}(Pr^R)}\mathcal{H}^n(-)_*$ and hence induces the same presheaf as \Cref{infty birational primary functoriality} (2). In particular, both assignments agree with the previous ones on Nisnevich covering diagrams and therefore satisfy Nisnevich descent. Since finite products in $\mathcal{H}^\infty(-)$ are computed levelwise (\Cref{Ln is cart} (1)), the formula $\mathrm{ev}_n\circ f^{\infty *}\simeq L^n f^*\circ\mathrm{ev}_{n+d}$ shows that pullbacks preserve finite products, giving the lift to $\mathrm{CAlg}(Pr^{L,\otimes})$ [\cite{lurie2017higher}, Corollary 2.4.1.8].

    Finally, consider the smooth base-change square of \Cref{ft bounded shift}. Since smooth extension is computed levelwise (\Cref{smooth extension pull}), applying $\mathrm{ev}_n$ to the base-change map at height $\infty$ gives
    \[
        p'_{n\shrp}\circ L^n f'^*\circ\mathrm{ev}_{n+d}
        \longrightarrow
        L^n f^*\circ p_{(n+d)\shrp}\circ\mathrm{ev}_{n+d},
    \]
    which is an equivalence by \Cref{ft bounded shift} (supplying the same $d$ for $f$ and $f'$). Since the evaluations are jointly conservative, smooth base change holds at height $\infty$, so [\cite{MR3701352}, Chapter 7, Theorem 3.2.2] yields the correspondence presheaf on $\mathrm{Cor}(\mathrm{N})_{ft,sm}^{op}$ as in \Cref{full functoriality}. The module case follows similarly from \Cref{sm proj formula} for $\mathcal{H}^\infty(-)$, and by localization for $L_{bir_\infty(-)}\mathcal{H}^{\mathbb{A}^1}(-)$.
\end{proof}
\begin{rem}
    Although each finite-height closed pushforward $i_*:\mathcal{H}^n(Z)\to \mathcal{H}^{n+scod(i)}(S)$ is fully faithful (\Cref{closed push full-faithful}), as is $i_*$ on $L_{bir_\infty}\mathcal{H}^{\mathbb{A}^1}$ (\Cref{infty birational functoriality} (1)), full faithfulness of $i_{\infty *}:\mathcal{H}^\infty(Z)\to \mathcal{H}^\infty(S)$ does not follow automatically because the forgetful functor $Pr^R\to \widehat{Cat}_\infty$ preserves limits rather than filtered colimits (cf.\ \Cref{2.2.3}).
\end{rem}